\documentclass[12pt]{amsart}
\usepackage{amssymb}
\usepackage[mathscr]{eucal}
\usepackage{epsf}
\usepackage{epsfig}
\usepackage{color}
\DeclareGraphicsExtensions{.pstex,.eps,.epsi,.ps}

 \newtheorem{thm}{Theorem}[section]
 \newtheorem{cor}[thm]{Corollary}
 \newtheorem{lem}[thm]{Lemma}
 \newtheorem{prop}[thm]{Proposition}
 \theoremstyle{definition}
 \newtheorem{defn}[thm]{Definition}
 \theoremstyle{remark}
 \newtheorem{rem}[thm]{Remark}
 \newtheorem{prob}[thm]{Problem}

\newcommand{\To}{\longrightarrow}
\newcommand{\e}{\epsilon}

\newcommand{\ue}{u_\e}
\newcommand{\Pe}[2]{P^\e_{#1,#2}}
\newcommand{\Pz}[2]{P^0_{#1,#2}}

\newcommand{\Real}{\mathbb{R}}

\newcommand{\LD}{\mathcal{L}}

\newcommand{\ddt}{\frac{d}{dt}}

\newcommand{\ddepZ}{\frac{d}{d\epsilon}\Big |_{\epsilon=0}}

\newcommand{\pddx}{\frac{\partial}{\partial x}}

\newcommand{\tl}[1]{\widetilde{#1}}
\newcommand{\br}[1]{\overline{#1}}

\newcommand{\C}{\mathfrak{C}}

\newcommand{\marginnote}[1]
{
}

\begin{document}

\title[Optimal Transportation under Nonholonomic Constraints]
{Optimal Transportation under Nonholonomic Constraints}

\author{Andrei Agrachev}
\email{agrachev@sissa.it}
\address{SISSA-ISAS, Trieste, Italy and MIAN, Moscow, Russia}

\author{Paul Lee}
\email{plee@math.toronto.edu}
\address{Department of Mathematics, University of Toronto, ON M5S 2E4, Canada}

\thanks{The authors were supported by PRIN (first author) and NSERC (second author) grants.}

\maketitle

\begin{abstract}

We study Monge's optimal transportation problem,  where the cost is
given by optimal control cost. We prove the existence and uniqueness
of an optimal map under certain regularity conditions on the
Lagrangian, absolute continuity of the measures with respect to Lebesgue, and most
importantly the absence of sharp abnormal minimizers. In particular,
this result is applicable in the case of subriemannian manifolds
with a 2-generating distribution and cost given by $d^2$, where $d$
is the subriemannian distance. Also, we discuss some properties of
the optimal plan when abnormal minimizers are present. Finally, we
consider some examples of displacement interpolation in the case of
Grushin plane.

\end{abstract}


\section{Introduction}

Let $(\mathcal X,\mu),\ (\mathcal Y,\nu)$ be probability spaces and
$c:\mathcal X\times\mathcal Y\to\Real\cup\{+\infty\}$ be a fixed
measurable function. The Monge's optimal transportation problem is
the minimization of the following functional
\[
\int_{\mathcal X}c(x,\phi(x))\,d\mu
\]
over all the Borel maps $\phi:\mathcal X\to \mathcal Y$ which pushes
forward $\mu$ to $\nu$: $\phi_*\mu=\nu$. Maps $\phi$ which achieve
the infimum above are called optimal maps. In this paper, we will
only consider the case when $\mathcal X=\mathcal Y=M$ is a manifold.

In 1942, Kantorovich studied a relaxed version of the Monge's problem
in his famous paper \cite{Ka}. However, a huge step toward solving
the original problem is not achieved until a decade ago by Brenier.
In \cite{Br}, Brenier proved the existence and uniqueness of optimal
map in the case where $M=\Real^n$ and the cost function $c$ is given
$c(x,y)=|x-y|^2$. Later, this is generalized to the case of a closed
Riemannian manifold $M$ with cost given by the square of the
Riemannian distance $c(x,y)=d^2(x,y)$ by McCann \cite{Mc}. Recently,
Bernard and Buffoni \cite{BeBu} generalized this further to the case
where the cost $c$ is the action associated to a Lagrangian function
$L:TM\to\Real$ on a compact manifold $M$. More precisely,
\begin{equation}\label{Lagcost}
c(x,y)=\inf_{x(0)=x,x(1)=y}\int_0^1L(x(t),\dot x(t))dt,
\end{equation}
where the infimum is taken over all curves joining the points $x$
and $y$, and the Lagrangian $L$ is fibrewise strictly convex with
superlinear growth.

In this paper, we consider costs similar to (\ref{Lagcost}).
However, instead of minimizing among all curves, the infimum is
taken over a subcollection of curves, called admissible paths. These
paths are given by a control system and the corresponding cost
function is called the optimal control cost. Roughly speaking,
a control system is a smooth fibrewise mapping of a locally trivial bundle
over $M$ into $TM$. Locally, such a mapping has a form
$F:(x,u)\mapsto F(x,u)$, where $x\in M,\ u\in U,\ F(x,u)\in T_xM $,
and $U$ is a typical fiber. We assume that $U$ is a closed subset of a
Euclidean space. Admissible controls are measurable bounded maps from
$[0,1]$ to $U$, while admissible paths are Lipschitz curves which
satisfy the equation
\begin{equation}\label{controlsys}
\dot x(t)=F(x(t),u(t)),
\end{equation}
where $u(\cdot)$ is an admissible control.
Let $L:M\times U\to\Real$ be a Lagrangian, then the corresponding cost $c$ is given by
\begin{equation}\label{controlcost}
\inf_{(x(\cdot),u(\cdot))}\int_0^1L(x(t),u(t))\,dt,
\end{equation}
where the infimum is taken over all admissible pairs
$(x(\cdot),u(\cdot)):[0,1]\to M\times U$ such that $x(0)=x,\ y(0)=y$.

In the interesting cases, the dimension of $U$ is essentially smaller
than that of $M$ and, nevertheless, any two points of $M$ can be connected
by an optimal admissible path. In other words, the control system works as
a nonholonomic constraint. Shortage of admissible velocities does not allow
to recover an optimal path from its initial point and initial velocity
and the Euler--Lagrange description of the extremals does not work well.
On the other hand, Hamiltonian approach remains efficient according to
the Pontryagin maximum principle. Another problem is appearance of so called
abnormal extremals (singularities of the space of admissible paths)
which we are obliged to fight with.

In sections 2 and 3, we will recall some basic notions in optimal
control theory and the theory of optimal mass transportation which
are necessary for this paper.

In section 4, by using some standard argument in the theory of
optimal mass transportation and the Pontryagin maximum principle in
optimal control theory, we show the existence and uniqueness of
optimal map under some regularity assumptions (Theorem \ref{main}).
All these conditions are mild except the Lipschitz continuity of the
cost function. However, this is well-known in all of the above cases
mentioned. So, the theorem generalizes the work in
\cite{Br,Mc,BeBu}.

In section 5, we study the Lipschitz continuity of the cost function.
If abnormal minimizers are absent, then the cost is not only Lipschitz
but even semi-concave (see \cite{CaRi}). Unfortunately, abnormal
minimizers are unavoidable in many interesting problems and, in particular,
in all subriemannian problems. It happens however that not all abnormal
minimizers are dangerous. To keep the Lipschitz property of the cost,
(though not the semi-concavity) it is sufficient that the, so called,
sharp abnormal minimizers are absent. Sharp paths are essentially
singularities of the space of admissible paths whose neighborhoods
in the second order approximations are contained in quadrics with
a finite Morse index. Geometric control theory provides simple effective
conditions of the sharpness (see, for instance, \cite{AgSa1, AgSa3}).
These conditions allow us to prove Lipschitz continuity for a large class
of optimal control cost. Hence, proving the existence and uniqueness of
optimal map of the corresponding Monge's problem (Theorem \ref{maintransport}).

In section 6, we apply the above results to some subriemannian
manifolds, where the cost function is given by the square of the
subriemannian distance (See section 6 for the basic notions in
subriemannian geometry). In the case of a subriemannian manifold,
all the mild regularity assumptions are satisfied. Using the result
in \cite{AgSa3} mentioned above (Proposition \ref{Goh}), Lipschitz
continuity of the cost can be easily proven in the case of a step 2
distribution (Corollary \ref{lip2}). Hence, proving existence and
uniqueness of optimal map. This generalizes the corresponding result
by Ambrosio and Rigot \cite{AmRi} on the Heisenberg group.

In section 7 and 8, we show some properties of the optimal plan when
abnormal minimizers are present. In section 7, we consider flows
whose trajectories are strictly abnormal minimizers. We show that
these flows cannot be an optimal plan for all ``nice" initial
measures if the cost is continuous. On the contrary, in section 8,
we show that these flows are indeed optimal for an important class
of problems with discontinuous cost.

In section 9, we study two examples on Grushin plane for  which the
results in section 3 and 4 apply.

\bigskip

\section{Elementary Optimal Control Theory}

In this section, we recall some notions from optimal control
theory. See \cite{AgSa1}, \cite{Ga} for detail. Let $M$ be a smooth manifold
and let $U$ be a closed subset in $\Real^m$ which is called the
control set. Let $F:M\times U\to TM$ be a Lipschitz continuous
function such that $F_u:=F(\cdot,u):M\to TM$ is a smooth vector
field for each point $u$ in the control set. Assume that the
function $(x,u)\mapsto \pddx F(x,u)$ is continuous. Curves $u(\cdot):[0,1]\to U$ in the
control set $U$ which are locally bounded and measurable are called
admissible controls.

A control system is the following ordinary
differential equations with parameters varying over all admissible
controls.

\begin{equation}\label{control}
\dot x(t)=F(x(t),u(t)).
\end{equation}
The solutions $x(t)$ to the above control system are called admissible
paths and $(x(t),u(t))$ are called admissible pairs.

By classical theory of ordinary differential equations, a
unique solution to the system (\ref{control}) exists locally for
almost all time $t$. Moreover, the resulting local flow is smooth
in the space variable $x$ and Lipschitz in the time variable $t$. The control system
is complete if the flows of all control vector fields exist globally.

Let $x_0$ and $x_1$ be two points on the manifold $M$. Denote by $\C_{x_0}$ the
set of all admissible pairs $(x(\cdot),u(\cdot))$ for which the corresponding admissible
paths $x(\cdot)$ start at the point $x_0$. And denote by $\C_{x_0}^{x_1}$ those pairs in
$\C_{x_0}$ whose admissible paths end at $x_1$. A control system is called controllable if the
set $\C_{x_0}^{x_1}$ is always nonempty for any pair of points
$x_0$ and $x_1$ on the manifold.

Let $L:M\times U\to\Real$ be a smooth function, called Lagrangian, and defined the cost
function corresponding to this Lagrangian as follow:

\begin{equation}\label{cost}
c(x_0,x_1)=\left\lbrace
  \begin{array}{c l}
    \inf\limits_{(x(\cdot),u(\cdot))\in\C_{x_0}^{x_1}}\int_0^1L(x(t),u(t))\,dt & \text{if $\C_{x_0}^{x_1}\neq\emptyset$},\\
    +\infty & \text{otherwise}.
  \end{array}
\right.
\end{equation}

The cost function defined above is said to be complete if given
any pairs of points $(x_0,x_1)$, there exists an admissible
pair which achieves the infimum above and the corresponding
admissible path starts from $x_0$ and ends at $x_1$.

\begin{rem}
The infimum of the problem in (\ref{cost}) can be equivalently
characterized by taking infimum over all admissible controls $u(\cdot)$
such that the corresponding admissible paths start at the point
$x_1$, end at the point $x_0$ of the manifold and satisfy the
following control system
\[
\dot x(s)=-F(x(s),u(s)).
\]
This point will become important for the later discussion.
\end{rem}

Consider the following minimization problem, commonly known as the
Bolza problem:

\begin{prob}\label{Bolza}
\[\inf_{(x(\cdot),u(\cdot))\in\C_{x_0}}\int_0^1L(x(s),u(s))\,ds-f(x(1))\]
\end{prob}

Next, we present an elementary version of the Pontryagin maximum
principle which we prove in the appendix for the convenience of
readers. Let $\pi:T^*M\to M$ be the cotangent bundle projection. For
each point $u$ in the control set $U$, define the corresponding
Hamiltonian function $H_u:T^*M\to\Real$ by
\[
H_u(p_x)=p_x(F(x,u))+L(x,u).
\]
If $H:T^*M\to\Real$ is a function on the cotangent bundle, we denote
its Hamiltonian vector field by $\overrightarrow{H}$.

\medskip

\begin{thm}\label{PMPB}(Pontryagin Maximum Principle for Bolza Problem)

Let $(\tl x(\cdot),\tl u(\cdot))$ be an admissible pair which achieves the
infimum in Problem \ref{Bolza}. Assume that the function $f$ in
Problem \ref{Bolza} is sub-differentiable at the point
$\tl x(1)$. Then, for each $\alpha$ in the sub-differential
$d^-f_{\tl x(1)}$ of $f$, there exists a Lipschitz path $\tl
p:[0,1]\to T^*M$ which satisfies the following for almost all time
$t$ in the interval $[0,1]$:

\begin{equation}
\left\{%
\begin{array}{ll}
    \pi(\tl p(t))=\tl x(t),\\
    \tl p(1)=-\alpha,\\
    \dot{\tl p}(t)=\overrightarrow H_{\tl u(t)}(\tl p(t)),\\
    H_{\tl u(t)}(\tl p(t))=\min\limits_{u\in U}H_u(\tl p(t)).

\end{array}%
\right.
\end{equation}

\end{thm}

\medskip

\begin{rem}\label{generalcontrol}

Let $\Delta\subset TM$ be a distribution on a $n$-dimensional manifold $M$. That is, for each point $x$ in the manifold $M$, it smoothly assigns a vector subspace $\Delta_x$ of the tangent space $T_xM$. Assume that the distribution $\Delta$ is trivializable, i.e. there exists a system of vector fields $X_1,...,X_k$ which span $\Delta$ at every point: $\Delta_x=\text{span}\{X_1(x),...,X_k(x)\}$.
Consider the following control system:

\begin{equation}\label{simplecontrol}
\dot x(t) = \sum_{i=1}^k u_i(t) X_i(x(t)),
\end{equation}
with initial condition $x(0)=x$ and final condition $x(1)=y$. Recall that we denote by $\C_x^y$ the set of all admissible pair $(x(\cdot),u(\cdot))$ such that the admissible path $x(\cdot)$ satisfies $x(0)=x$ and $x(1)=y$. Let $c$ be the cost given by

\begin{equation}\label{simplecost}
c(x,y) = \inf_{(x(\cdot),u(\cdot))\in\C_x^y} \int_0^1 \sum_{i=1}^k u_i^2 \,dt.
\end{equation}
If $k$ is equal to the dimension $n$ of the manifold $M$, then the distribution $\Delta$ is the same as the tangent bundle $TM$ of $M$ and the admissible paths of the control system (\ref{simplecontrol}) are all the paths on $M$. It also defines a Riemannian metric on $M$ by declaring that the vector fields $X_1,...,X_n$ are orthonormal everywhere. The cost function $c$ is the square of the Riemannian distance $d$: $c=d^2$. And the minimizers of this system correspond to the length minimizing geodesics on $M$. However, this does not work for distributions which are not trivializable. (For instance, any vector field on the $2$-sphere has a zero.)

To overcome this difficulty, we can modify the general definition of control system in the following way.
Let $V$ be a locally trivial bundle on $M$ with bundle projection $\pi^V:V\to M$ and let $F:V\to TM$ be fibre preserving map. i.e. $F(V_x)\subseteq T_xM$.
The control system corresponding to the map $F$ is given by

\begin{equation}\label{generalcontrol}
\dot x(t)=F(v(t)).
\end{equation}
The admissible pairs are locally bounded measurable paths $v(\cdot):[0,1]\to V$ in $V$ such that its projection to the manifold $M$ is a Lipschitz path: $x(\cdot)=\pi^V(v(\cdot))$ is Lipschitz. If we let $V$ be the trivial bundle $M\times U$, we get back the system (\ref{control}). If a Lagrangian $L:V\to\Real$ is fixed, then the corresponding cost function $c$ is defined by

\begin{equation}\label{generalcost}
c(x,y)=\inf_{v(\cdot)\in\C_x^y} \int_0^1 L(v(t)) dt,
\end{equation}
where the infimum is taken over all admissible pair $v(\cdot):[0,1]\to V$ such that the corresponding admissible path $x(\cdot)=\pi^V(v(\cdot))$ satisfies $x(0)=x$ and $x(1)=y$.

Let $<,>$ be a Riemannian metric on the manifold $M$. If $V$ is the tangent bundle $TM$ of $M$, the map $F$ is the identity map and the Lagrangian $L:V\to\Real$ given by $L(v)=<v,v>$, then the cost function $c$ is equal to the square of the Riemannian distance. If $k < n$, then the admissible paths of the control system (\ref{simplecontrol}) are paths tangent to the distribution $\Delta$. Similar to the Riemannian case, the control system defines a subriemannian metric $<,>^S$. (See section 6 for the basics on subriemannian geometry) And the cost (\ref{simplecost}) is the square of the subriemannian distance $d_S$: $c=d_S^2$. For general distributions $\Delta$ which are not trivializable, consider the general control system (\ref{generalcontrol}) with $V=\Delta$. And $F:\Delta\hookrightarrow TM$ is the inclusion map. If the Lagrangian $L$ is defined by $L(v)=<v,v>^S$, then the cost is again the square of the subriemannian distance.

In this paper (except in section 8), we consider the control systems of the form (\ref{control}) in order to avoid heavy notations. All the results have easy generalization to more general intrinsically defined systems just introduced.

\end{rem}

\bigskip

\section{Optimal Mass Transport}

The theory of optimal mass transportation is about moving one mass
to another that minimizes certain cost. More precisely, let $M$ be
a manifold and consider a function $c:M\times M\to
\Real\cup\{+\infty\}$, called the cost function. Let $\mu$ and $\nu$ be two Borel probability measures on the manifold $M$, then the optimal mass transportation is the following problem:

\begin{prob}\label{optimal}

Find a Borel map which achieves the following infimum among all Borel maps $\phi:M\to M$ that pushes the probability measure $\mu$ forward to $\nu$

\[\inf_{\phi_*\mu=\nu}\int_M c(x,\phi(x)) \,d\mu.\]

\end{prob}

Here, we recall that the push forward $\phi_*\mu$ is defined by $\phi_*\mu(B)=\mu(\phi^{-1}(B))$ for all Borel set $B$ in $M$. In many cases, such a problem admits solution which is unique (up
to measure zero), assuming absolute continuity of the measure $\mu$ with respect to the Lebesgue measure. This unique solution to (\ref{optimal}) is called the optimal map or the Brenier map.

The first optimal map was found by Brenier in \cite{Br} in the case where the manifold is $\Real^n$ and the cost was $c(x,y)=|x-y|^2$. Later, it was generalized to arbitrary closed, connected Riemannian manifolds in \cite{Mc} with cost given by square of the Riemannian distance. The case for the Heisenberg group with the cost given by $d^2$ was done in \cite{AmRi}, where $d$ is the subriemannian distance or the gauge distance. In \cite{BeBu}, a much general cost given by the action associated to a Lagrangian function $L:TM\to\Real$ on a compact manifold $M$ was considered. More precisely,

\begin{equation}\label{Lag}
c(x,y)=\inf_{x(0)=x,x(1)=y}\int_0^1L(x(t),\dot x(t))dt,
\end{equation}
where the infimum is taken over all curves joining the points $x$ and $y$.

Existence and uniqueness of optimal map with the cost given by (\ref{Lag}) was shown under the following assumptions:
\begin{itemize}
  \item The Lagrangian $L$ is fibrewise strictly convex, i.e. the map restriction of $L$ to the tangent space $T_xM$ is strictly convex for each fixed $x$ in the manifold $M$.
  \item $L$ has superlinear growth, i.e. $L(v)/|v|\to 0$ as $|v|\to\infty$.
  \item The cost $c$ is complete, i.e. the infimum (\ref{Lag}) is always achieved by some $C^2$ smooth paths.
\end{itemize}
Recently, the compactness assumption on the manifold or on the measures was eliminated \cite{Fi,FaFi}.

In this paper, we consider a connected manifold $M$ without boundary and the cost function
$c$ is given by (\ref{cost}). Consider the following relaxed version of Problem
\ref{optimal}, called Kantorovich reformulation. Let
$\pi_1:M\times M\to M$ and $\pi_2:M\times M\to M$ be the
projection onto the first and the second component respectively.
Let $\Gamma$ be the set of all joint measures $\Pi$ on the product
manifold $M\times M$ with marginals $\mu$ and $\nu$: $\pi_{1*}\Pi=\mu$ and
$\pi_{2*}\Pi=\nu$.

\begin{prob}\label{KanRe}
\[
C(\mu,\nu):=\inf_{\Pi\in\Gamma}\int_{M\times M}c(x,y)\, d\Pi(x,y)
\]
\end{prob}

\begin{rem}
If $\phi$ is an optimal map in the problem in (\ref{optimal}),
then $(id\times\phi)_*\mu$ is a joint measure in the set $\Gamma$.
Therefore, Problem \ref{KanRe} is a relaxation of the problem in
(\ref{optimal}).
\end{rem}

Before we proceed into the existence proof of optimal map, let us look at
the following dual problem of Kantorovich. See \cite{Vi1} for
history and the importance of this dual problem to optimal
transport.

Let $c$ be a cost function and let $f$ be a function on the manifold $M$. The
$c_1$-transform of the function $f$ is the function $f^{c_1}$
given by
\[
f^{c_1}(y):=\inf_{x\in M}[c(x,y)-f(x)].
\]

Similarly, the
$c_2$-transform of the function $f$ is given by
\[f^{c_2}(x):=\inf_{y\in M}[c(x,y)-f(y)].\] The function $f$ is a $c$-concave
function if it satisfies $f^{c_1c_2}=f$. Let $\mathfrak{F}$ be the
set of all pairs of functions $(g,h)$ on the manifold such that $g:M\to\Real\cup\{-\infty\}$ and  $h:M\to\Real\cup\{-\infty\}$ are in $L^1(\mu)$ and $L^1(\nu)$ respectively, and
$g(x)+h(y)\leq c(x,y)$ for all $(x,y)\in M\times M$. The dual problem of Kantorovich is the
following maximization problem:

\begin{prob}\label{dual}
\[
\sup_{(g,h)\in\mathfrak{F}}\int_M gd\mu+\int_M h\,d\nu.
\]
\end{prob}

The existence of solution to the above problem is well-known. See \cite{Vi1} and \cite{Vi2} for the proof.

\medskip

\begin{thm}\label{dualexist}
Assume that there exists two functions $c_1$ and $c_2$ such that $c_1$ is $\mu$-measurable, $c_2$ is $\nu$-measurable and the cost function $c$ satisfies $c(x,y)\leq c_1(x)+c_2(y)$ for all $(x,y)$ in $M\times M$. If $c$ is also continuous, bounded below and $C(\mu,\nu)<\infty$, then there exists a $c$-concave function $f$ such that the function $f$ is in $L^1(\mu)$, its $c_1$-transform $f^{c_1}$ is in $L^1(\nu)$ and the pair $(f,f^{c_1})$ achieves the supremum in Problem \ref{dual}.
\end{thm}

\medskip

The following theorem on the regularity of the dual pair above is also well-known.

\begin{thm}\label{dualregular}
Assume that the cost $c(x,y)$ is continuous, bounded below and the measures $\mu$ and $\nu$ are compactly supported. Then the functions $f$ and $f^{c_1}$ are upper semicontinuous. If the function $x\mapsto c(x,y)$ is also locally Lipschitz on an set $\mathcal U$ and the Lipschitz constant is independent of $y$ locally, then  $f$ can be chosen to be locally Lipschitz on $\mathcal U$.
\end{thm}

\begin{proof}

Fix $\e>0$. Since $f(x)=\inf_{x\in M}[c(x,y)-f^{c_1}(y)]$, there exists $y$ such that $f(x)+\e/2>c(x,y)-f^{c_1}(y)$. Also, we have $f(x')+f^{c_1}(y)\leq c(x',y)$ for any $x'$ in $M$. So, combining the above equations and continuity of the cost $c$, we have
\[
f(x')-f(x)<\e
\]
for any $x'$ close enough to $x$. Therefore, $f$ is upper semicontinuous.

Let $K$ be a compact set containing the support of the measures $\mu$ and $\nu$.
Let $$g(x)=\left\{
         \begin{array}{ll}
           f(x), & \hbox{\text{if } $x\in K$} \\
           -\infty, & \hbox{\text{if } $x\in M\setminus K$}
         \end{array}
       \right.,\ g'(x)=\left\{
         \begin{array}{ll}
           f^{c_1}(x), & \hbox{\text{if } $x\in K$} \\
           -\infty, & \hbox{\text{if } $x\in M\setminus K$}
         \end{array}
       \right.,$$
then the pair $(g,g')$ achieves the maximum in Problem \ref{dual}. Let $h=(g')^{c_2}$,
then the pair $(h,h^{c_1})$ also achieves the maximum. By definition of $g'$, we
have $h(x)=\inf\limits_{y\in K}[c(x,y)-f^{c_1}(y)]$.
By an argument the same as the proof of upper semicontinuity, for any $x$ and $x'$ in the compact subset $K'$ of $\mathcal U$, we can find $y$ in $K$ such that
\[
h(x')-h(x)<c(x,y)-c(x',y)+\e/2.
\]
By the assumption of the cost $c$, the above inequality becomes
\[
h(x')-h(x)\leq kd(x,x')+\e/2
\]
for some constant $k > 0$ which is independent of $x$ on $K'$. By switching the roles of $x$ and $x'$, the result follows.
\end{proof}

\medskip

The following theorem about minimizers of the Problem \ref{KanRe} is well-known. See, for instance, \cite{Vi1}.

\medskip

\begin{thm}\label{existjoint}
If we make the same assumption as in Theorem \ref{dualexist}, then Problem \ref{KanRe} admits a minimizer. Moreover, the joint measure $\Pi$ in the set $\Gamma$ achieve the infimum in Problem \ref{KanRe} if and only if $\Pi$ is concentrated on the set

\[
\{(x,y)\in M\times M|f(x)+f^c(y)=c(x,y)\}.
\]

\end{thm}

\bigskip

\section{Existence and Uniqueness of Optimal Map}

In this section, we show that Monge's problem with cost given by an optimal control cost (\ref{controlcost}) can be solved under certain regularity assumptions. Let $H:T^*M\to\Real$ be the function defined by
\[
H(p_x)=\max_{u\in U}\left(p_x(F(x,u))-L(x,u)\right).
\]
If $H$ is well-defined and $C^2$, then we denote its Hamiltonian vector field by $\overrightarrow{H}$ and let $e^{t\overrightarrow{H}}$ be its flow. Let $f$ be the function defined in Theorem \ref{dualexist} which is Lipschitz for $\mu$-almost all $x$. Consider the map $\varphi:M\times[0,1]\to M$ defined by $\varphi(x,t)=\pi(e^{t\overrightarrow{H}}(-df_x))$.

\begin{thm}\label{main}
The map $x\mapsto\varphi_1(x):=\varphi(x,1)$ is the unique (up to $\mu$-measure zero) optimal map to the problem (\ref{optimal}) with cost $c$ given by (\ref{cost}) under the following assumptions:

\begin{enumerate}
  \item The measures $\mu$ and $\nu$ are compactly supported and absolutely continuous with respect to the Lebesgue measure.
  \item $c$ is bounded below and $c(x,y)$ is also locally Lipschitz in the $x$ variable and the Lipschitz constant is independent of $y$ locally.
  \item The cost $c$ is complete, i.e. given any pairs of points $(x_0,x_1)$ in the manifold $M$, there exists an admissible pair $(x(\cdot),u(\cdot))$ such that the pair achieves the infimum in (\ref{cost}), $u(\cdot)$ is locally bounded measurable, $x(0)=x_0$ and $x(1)=x_1$.
  \item The Hamiltonian function $H$ defined in (\ref{maxHam}) is well-defined and $C^2$.
  \item The Hamiltonian vector field $\overrightarrow{H}$ is complete, i.e. global flow exists.
\end{enumerate}
\end{thm}

The rest of this section is devoted to the proof of Theorem \ref{main}. Let $\tl \C_y$ be the set of all admissible pairs such that the corresponding admissible paths $x(\cdot)$ starts from the point $y$: $x(0)=y$ and satisfies the following control system:

\begin{equation}\label{backcontrol}
\dot x(t)=-F(x(t),u(t)).
\end{equation}

Let $\tl \C_y^x$ be the set of all those pairs in $\tl \C_y$ such that the corresponding admissible paths $x(\cdot)$ end at the point $x$: $x(1)=x$.

First, we have the following simple observation.

\medskip

\begin{prop}\label{dualmin}
Let $x$ be a point which achieves the infimum
$f^{c_1}(y)=\inf\limits_{x\in M}\left(c(x,y)-f(x)\right)$ and let $(\tl x,\tl u)$ be an admissible pair in $\tl \C_y^x$ such that the corresponding admissible path $\tl x$ minimizes the cost given by

\[
c(x,y)=\inf_{(x(\cdot),u(\cdot))\in\tl \C_y^x}\int_0^1L(x(t),u(t))\,dt,
\]
then $(\tl x(\cdot),\tl u(\cdot))$ achieves the following infimum

\begin{equation}\label{dualmin1}
f^{c_1}(y)=\inf_{(x(\cdot),u(\cdot))\in\tl \C_y}\int_0^1L(x(s),u(s))\,ds-f(x(1)).
\end{equation}

If $\hat x(t)=\tl x(1-t)$, then $\hat x$ achieves the following infimum
\begin{equation}\label{dualmin2}
f^{c_1}(y)=\inf_{(x(\cdot),u(\cdot))\in\C^y}\int_0^1L(x(s),u(s))\,ds-f(x(0)),
\end{equation}
where $\C^y$ denotes the set of all admissible pairs $(x(\cdot),u(\cdot))$ satisfying the following control system:
\[
\dot x(t)=F(x(t),u(t)),\quad x(1)=y.
\]

\end{prop}

\medskip

Let $\tl u(\cdot)$ be as in the above Proposition and let $\hat u(t)=\tl u(1-t)$. Let $H_t:T^*M\to\Real$ be given by $H_t(p_x)=p_x(F(x,\hat u(t)))-L(x,\hat u(t))$. The following is a consequence of Theorem \ref{PMPB}.

\begin{prop}\label{appPMP}
Let $\tl x$ be a curve that achieves the infimum in (\ref{dualmin1}) and let $\hat x(t)=\tl x(1-t)$. Assume that $\alpha$ is contained in the subdifferential of the function $f$ at the point $\hat x(0)$,
then there exists a Lipschitz curve $\hat p:[0,1]\to T^*M$ in the cotangent bundle such that the followings are true for almost all time $t$ in the interval $[0,1]$:
\begin{equation}
\left\{%
\begin{array}{ll}
    \pi(\hat p(t))=\hat x(t),\\
    \dot {\hat p}(t)=\overrightarrow H_t(\hat p(t)),\\
    \hat p(0)=-\alpha,\\
    H_t(\hat p(t))=\max\limits_{u\in U}\left(\hat p(t)(F(\hat x(t),u))-L(\hat x(t),u)\right)
\end{array}%
\right.
\end{equation}
\end{prop}

\begin{proof}
By Theorem \ref{PMPB}, there exists a curve $\tl p:[0,1]\to T^*M$ in the cotangent bundle $T^*M$ such that

\[
\left\{%
\begin{array}{ll}
    \pi(\tl p(t))=\tl x(t),\\
    \tl p(1)=-\alpha,\\
    \dot{\tl p}(t)=\overrightarrow{\tl H}_{\tl u(t)}(\tl p(t)),\\
    \tl H_{\tl u(t)}(\tl p(t))=\min\limits_{u\in U}\left(-\tl p(t)(F(\tl x(t),\tl u(t)))+L(\tl x(t),\tl u(t))\right),
\end{array}%
\right.
\]
where $\tl H_{\tl u}(p)=\min\limits_{u\in U} (-\tl p(F(\tl x,\tl u(t)))+L(\tl x,\tl u(t)))$.

Let $\hat p(t)=\tl p(1-t)$ and $\hat u(t)=\tl u(1-t)$, then the equations above become

\[
\left\{%
\begin{array}{ll}
    \pi(\hat p(t))=\hat x(t),\\
    \hat p(0)=-\alpha,\\
    \dot{\hat p}(t)=\overrightarrow H_{\hat u(t)}(\hat p(t)),\\
    H_{\hat u(t)}(\hat p(t))=
    \max\limits_{u\in U}\left(\hat p(t)(F(\hat x(t),\hat u(t)))-L(\hat x(t),\hat u(t))\right).
\end{array}%
\right.
\]

\end{proof}

\medskip

Assume that the Hamiltonian function $H:T^*M\to\Real$ defined by
\begin{equation}\label{maxHam}
H(p_x)=\max_{u\in U}\left(p_x(F(x,u))-L(x,u)\right)
\end{equation}
is well-defined and $C^2$. Let $\overrightarrow{H}$ be the Hamiltonian vector
field of the function $H$ and let $e^{t\overrightarrow H}$ be its flow. The function $f$ defined in
Theorem \ref{dualexist} is Lipschitz and so it is differentiable
almost everywhere by Rademacher Theorem. Moreover, the map $df:M\to T^*M$ is measurable and locally
bounded. So, if we let $\varphi:M\times[0,1]\to M$ be the map defined by
$\varphi(x,t)=\pi(e^{t\overrightarrow{H}}(-df_x))$, then the map $\varphi$ is a Borel map.

\medskip

\begin{prop}\label{dualunique}
Under the assumptions of Theorem \ref{main}, the following is true for $\mu$-almost all $x$: Given a point $x$ in the support of $\mu$, there exists a unique point $y$ such that

\[
f(x)+f^{c_1}(y)=c(x,y).
\]
Moreover, the points $x$ and $y$ are related by $y=\varphi(x,1)$.
\end{prop}

\medskip

\begin{proof}
We first claim that the infimum $f(x)=\inf_{y\in M}[c(x,y)-f^{c_1}(y)]$ is achieved for $\mu$ almost all $x$. Indeed, by assumption, we have $f(x)+f^{c_1}(y)\leq c(x,y)$ for all $(x,y)$ in $M\times M$. Also, let $\Pi$ be the measure defined in Theorem \ref{existjoint}, then $f(x)+f^{c_1}(y)= c(x,y)$ for $\Pi$-almost everywhere. Since the first marginal of the measure $\Pi$ is $\mu$, the following is true for $\mu$ almost all $x$: Given a point  $x$ in the manifold $M$, there exists $y$ in $M$ such that $f(x)+f^{c_1}(y)= c(x,y)$. This proves the claim.

Fix a point $x$ for which the infimum $\inf_{y\in M}[c(x,y)-f^{c_1}(y)]$ is achieved and let $y$ be the point which achieves the infimum. By the proof of the above claim, $x$ achieves the infimum $f^{c_1}(y)=\inf_{x\in M}[c(x,y)-f(x)]$. Therefore, by completeness of the cost $c$ and Proposition \ref{dualmin}, there exists an admissible path $\hat x$ such that $\hat x(0)=x$, $\hat x(1)=y$ and $\hat x$ achieves the infimum (\ref{dualmin2}).

Since $f$ is Lipschitz on a bounded open set $\mathcal U$ containing the support of $\mu$ and $\nu$, it is almost everywhere differentiable on $\mathcal U$ by Rademacher Theorem. Since $\mu$ is absolutely continuous with respect to the Lebesgue measure, $f$ is also differentiable $\mu$-almost everywhere. By Theorem \ref{appPMP}, for $\mu$-almost all $x$, there exists a curve $\hat p:[0,1]\to T^*M$ in the cotangent bundle $T^*M$ such that
\[
\left\{%
\begin{array}{ll}
    \dot {\hat p}(t)=\overrightarrow H_t(\hat p(t)),\\
    \hat p(0)=-df_x,\\
    \pi(\hat p(t))=\hat x(t),\\
    H_t(\hat p(t))=\max\limits_{u\in U}\left(\hat p(t)(F(\hat x(t),u))-L(\hat x(t),u)\right),
\end{array}%
\right.
\]
where $H_t$ is the function on the cotangent bundle $T^*M$ given by $H_t(p_x)=p_xF(x,u(t))-L(x,u(t))$.

By the definition of $H$, we have $H(\hat p(t))=H_t(\hat p(t))$. But, we
also have $H(p)\geq H_t(p)$ for all $p\in T^*M$. Since both $H$ and $H_t$ are $C^2$,
we have $dH(\hat p(t))=dH_t(\hat p(t))$. Hence,
$\overrightarrow H_t(\hat p(t))=\overrightarrow{H}(\hat p(t))$ for almost all $t$. The result
follows from uniqueness of solution to ODE.
\end{proof}

The rest of the arguments for the existence and uniqueness of optimal map
follow from Theorem \ref{existjoint}.

\medskip

\begin{proof}[Proof of Theorem \ref{main}]
As mentioned above, Problem \ref{KanRe} is a relaxation of Problem \ref{optimal}. We can
recover the later from the former by restricting the minimization to joint measures of
the form $(id\times\phi)_*\mu$, where $\phi$ is any Borel map pushing forward $\mu$ to
$\nu$. Therefore, the results follow from Theorem \ref{existjoint} and
Proposition \ref{dualunique}.
\end{proof}

\bigskip

\section{Regularity of Control Costs}

In Theorem \ref{main}, we prove existence and uniqueness of optimal maps under certain regularity conditions on the cost. Most of the conditions in the theorem are easy to verify except condition (2) and (3). In this section, we will give simple conditions which guarantee this regularity. This includes the completeness and the Lipschitz regularity of the cost. First, we recall some basic notions in the geometry of optimal control problems, see \cite{Ag} and reference therein for details.

Fix a point $x_0$ on the manifold $M$ and assume that the control set $U$ is $\Real^k$. Denote by $\C_{x_0}$ the set of all admissible pairs $(x(\cdot),u(\cdot))$ such that the corresponding admissible paths $x(\cdot)$ starts at $x_0$. Moreover, we assume that the control system is of the following form:

\begin{equation}\label{driftcontrol}
\dot x(t) = X_0(x(t)) + \sum_{i=1}^k u_i(t) X_i(x(t)),
\end{equation}
where $u(t)=(u_1(t),...,u_k(t))$ and $X_0,X_1,...,X_k$ are fixed
smooth vector fields on the manifold $M$. The Cauchy problem for
system (\ref{driftcontrol}) is correctly stated for any locally
integrable vector-function $u(\cdot)$.  We assume, throughout
this section, that system (\ref{driftcontrol}) is complete,
i.\,e. all solutions of the system are defined on the whole
semi-axis $[0,+\infty)$. This completeness assumption is
automatically satisfied if one of the following is true: (i) if
$M$ is a compact manifold, (ii) $M$ is a Lie group and the fields
$X_i$ are left-invariant, or (iii) if $M$ is a closed submanifold
of the Euclidean space and $|X_i(x)|\le c(1+|x|),\ i=0,1,\ldots
k$.

Define the endpoint map $End_{x_0}:L^\infty([0,1],\Real^k)\to M$ by
\[
End_{x_0}(u(\cdot))=x(1),
\]
where $(x(\cdot),u(\cdot))$ is the admissible pair corresponding to the control system (\ref{driftcontrol}) with initial condition $x(0)=x_0$. It is known that the map $End_{x_0}$ is a smooth mapping. The critical points of the map $End_x$ are called \emph{singular controls}. Admissible paths corresponding to singular controls are called \emph{singular paths}.

We also need the Hessian of the mapping $End_{x_0}$ at the critical point. (See
\cite{AgSa1} for detail.) Let $E$ be a Banach space which is an
everywhere dense subspace of a Hilbert space $H$. Consider a mapping
$\Phi:E\to\Real^n$ such that the restriction of this map $\Phi\bigr|_W$ to any
finite dimensional subspace $W$ of the Banach space $E$ is $C^2$.
Moreover, we assume that the first and second derivatives of all
the restrictions $\Phi\bigr|_W$ are continuous in the Hilbert space topology on
the bounded subsets of $E$. In other words,
$$
\Phi(v+w)-\Phi(v)=D_v\Phi(w)+\frac 12D^2_v\Phi(w)+o(|w|^2),\ w\in W,
$$
where $D_v\Phi$ is a linear map and $D^2_v\Phi$ is a quadratic mapping from $E$ to
$\mathbb R^n$. Moreover, $\Phi(v)$, $D_v\Phi\bigr|_W$ and
$D^2_v\Phi\bigr|_W$ depend continuously on $v$ in the topology of $H$ while $v$ is contained in a ball of $E$.

The Hessian $\mathrm{Hess}_v\Phi:\ker D_v\Phi\to \mathrm{coker}
D_v\Phi$ of the function $\Phi$ is the restriction of $D^2_v\Phi$
to the kernel of $D_v\Phi$ with values considered up to the image
of $D_v\Phi$. Hessian is a part of $D^2_v\Phi$ which survives
smooth changes of variables in $E$ and $\mathbb R^n$.

Let $p$ be a covector in the dual space $\Real^{n*}$ such that $pD_v\Phi=0$,
then $p\mathrm{Hess}_v\Phi$ is a well-defined real quadratic form
on $\ker D_v\Phi$. We denote the Morse index of this quadratic
form by $\mathrm{ind}(p\mathrm{Hess}_v\Phi)$. Recall that the
Morse index of a quadratic form is the supremum of dimensions of
the subspaces where the form is negative definite.

\begin{defn}
A critical point $v$ of $\Phi$ is called {\it sharp}
if there exists a covector $p\ne 0$ such that $pD_v\Phi=0$ and
$\mathrm{ind}(p\mathrm{Hess}_v\Phi)<+\infty$.
\end{defn}
Needless to say, the spaces $E,H$ and $\mathbb R^n$ can be substituted by smooth
manifolds (Banach, Hilbert and $n$-dimensional) in all this terminology.

Going back to the control system (\ref{driftcontrol}), let
$(x(\cdot),u(\cdot))$ be an admissible pair for this system. We
say that the control $u(\cdot)$ and the path $x(\cdot)$ are sharp
if $u(\cdot)$ is a sharp critical point of the mapping
$End_{x(0)}$.

One necessary condition for controls and paths to be sharp is the, so called, Goh condition.

\begin{prop}\label{Goh}(Goh condition) If
$p(Hess_{u(\cdot)}End_{x(0)})<+\infty$, then
\[
p(t)(X_i(x(t)))=p(t)([X_i,X_j](x(t)))=0 \quad i,j=1,\ldots,k,\ 0\le t\le 1,
\]
where $p(t)=P_{t,1}^*p$ and $P_{t,\tau}$ is the local flow of the control system (20)
with control equal to $u(\cdot)$.
\end{prop}
See \cite{AgSa1,AgSa3} and references therein for the proof and other effective necessary
and sufficient conditions of the sharpness.

Now consider the optimal control problem
\begin{equation}\label{costagain}
c(x,y)=\inf_{(x(\cdot),u(\cdot))\in\C_x^y}\int_0^1L(x(t),u(t))\,dt,
\end{equation}
where the infimum ranges over all admissible pairs $(x(t),u(t))$ corresponding to the control system (\ref{driftcontrol}) with initial condition $x(0)=x$ and final condition $x(1)=y$.

Let $H:T^*M\to\Real$ be the Hamiltonian function defined in (\ref{maxHam}). For
simplicity, we assume that the Hamiltonian is $C^2$. A minimizer $x(\cdot)$ of the above
problem is called normal if there exists a curve $p:[0,1]\to T^*M$ in the cotangent
bundle $T^*M$ such that $\pi(p(t))=x(t)$ and $p(\cdot)$ is a trajectory of the
Hamiltonian vector field $\overrightarrow{H}$. Singular minimizers are also called abnormal.
According to this not perfect terminology a minimizer can simultaneously be normal and abnormal.
A minimizer which is not normal is called strictly abnormal.

The following theorem gives simple sufficient conditions for completeness of the cost function defined in (\ref{costagain}). It is a combination of the well-known existence result (see \cite{SaTo}) and necessary optimality conditions (see \cite{AgSa1}).

\medskip

\begin{thm}(Completeness of costs)\label{mintraj}
Let $L$ be a Lagrangian function which satisfies the following:
\begin{enumerate}
\item $L$ is bounded below and
 there exist constants $K>0$ such that the ratio $\frac{|u|}{L(x,u)+K}$ tends to 0 as $|u|\to\infty$ uniformly on compact subsets of $M$;
\item for any compact $C\subset M$
there exist constants $a, b>0$ such that $|\frac{\partial L}{\partial x}(x,u)| \leq a (L(x,u)+|u|)+b$,
$\forall x\in C,\ u\mathbb R^k$;
\item the function $u\mapsto L(x,u)$ is a strongly convex function for all $x\in M$.
\end{enumerate}
Then, for each pair of points $(x,y)$ in the manifold $M$ which satisfy $c(x,y)<+\infty$, there exists an admissible pair $(x(\cdot),u(\cdot))$ achieving the infimum in (\ref{costagain}). Moreover, the minimizer $x(\cdot)$ is either a normal or a sharp path.
\end{thm}

\medskip

\begin{rem}
Theorem \ref{mintraj} gives lots of examples that satisfy condition (3) in Theorem \ref{main}. In particular, this applies to the case where the control set $U=\Real^k$ and the Lagrangian is $L(x,u)=\sum_{i=1}^k u_i^2$.
\end{rem}

Next, we proceed to the main result of of this section which concerns with the Lipschitz regularity of the cost function. This takes care of the condition (2) in Theorem \ref{main}.

\medskip

\begin{thm}\label{lip}(Lipschitz regularity)
Assume that the system (\ref{simplecontrol}) does not admit sharp
controls and the Lagrangian $L$ satisfies conditions of Theorem \ref{mintraj}, then the set $D=\{(x,End_x(u(\cdot)))|x\in M,u\in
L^\infty([0,1],\Real^k)\}$ is open in the product $M\times M$.
Moreover, the function $(x,y)\mapsto c(x,y)$ is locally
Lipschitz on the set $D$, where the cost $c$ is given by (\ref{cost}).
\end{thm}

\medskip

\begin{rem}
In the case where the endpoint map is a submersion, there is no singular control. Therefore, Theorem \ref{lip} is applicable. In particular, this theorem, together with Theorem \ref{main} and \ref{mintraj}, can be used to treat the cases considered in \cite{Br,Mc,BeBu}. In section 5, we will consider a class of examples where the endpoint map is not necessarily a submersion, but Theorem \ref{lip} is still applicable.
\end{rem}

The rest of the section is devoted to the proof of Theorem \ref{lip}.

\begin{defn}

Given $v$ in the Banach space $E$, we write $\mathrm{Ind}_v\Phi\ge m$ if
$$
\mathrm{ind}(p\mathrm{Hess}_v\Phi)-\mathrm{codim\,im}D_v\Phi\ge m
$$
for any $p$ in $\mathbb R^{n*}\setminus\{0\}$ such that $pD_v\Phi=0$.
\end{defn}

It is easy to see that $\{v\in E : \mathrm{Ind}_v\Phi\ge m\}$ is an open subset of $E$ for any integer $m$. Let $\mathbf B_v(\varepsilon),\,B_x(\varepsilon)$ be radius $\varepsilon$ balls in $E$ and $\mathbb R^n$ centered at $v$ and $x$ respectively. The following is a qualitative version of openness of a mapping $\Phi$ and any mapping $C^0$ close to it.

\begin{defn} We say that the map $\Phi:E\to\Real^n$ is $r$-solid at the point $v$ of the Banach space $E$ if
for some constant $c>0$ and any sufficiently small $\varepsilon>0$, there exists $\delta>0$ such that $\tilde\Phi(\mathbf
B_v(\varepsilon))\supset B_{\tilde\Phi(v)}(c\varepsilon^r)$ for
any $\tilde\Phi:\mathbf B_v(\varepsilon)\to\Real^n$ such that $\sup\limits_{w\in\mathbf
B_v(\varepsilon)}|\tilde\Phi(w)-\Phi(w)|\le\delta$.
\end{defn}

Implicit function theorem together with Brouwer fixed point theorem imply that $\Phi$ is 1-solid at any regular point.

\medskip

\begin{lem}\label{2solid}
If $\mathrm{Ind}_v\Phi\ge 0$ then $\Phi$ is 2-solid at $v$.
\end{lem}

\medskip

\begin{proof}
This lemma is a refinement of Theorem 20.3 from
\cite{AgSa1}. It can be proved by a slight modification of the proof
of the cited theorem. Obviously, we may assume that $v$ is a
critical point of $\Phi$. Moreover, by an argument in the proof of the above cited theorem, we may assume that $E$ is a finite dimensional space, $v=0$ and $\Phi(0)=0$.

Let $E=E_1\oplus E_2$, where $E_2=\ker D_0\Phi$. For any $w\in E$ we write $w=w_1+w_2$, where $w_1\in E_1,\ w_2\in E_2$. Now consider the mapping
$$
Q:v\mapsto D_0\Phi v_1+\frac 12D^2_0\Phi(v_2),\ v\in E.
$$
It is shown in the proof of Theorem 20.3 from \cite{AgSa1} that
$Q^{-1}(0)$ contains regular points in any neighborhood of 0.
Hence $\exists\, c>0$ such that the image of any continuous mapping
$\tilde Q:\mathbf B_0(1)\to\mathbb R^n$ sufficiently close (in
$C^0$-norm) to $Q\bigr|_{\mathbf B_0(1)}$ contains $B_0(c)$. Now
we set $\Phi_\varepsilon(v)=\frac
1{\varepsilon^2}\Phi(\varepsilon^2v_1+\varepsilon v_2)$; then
$\Phi_\varepsilon(v)=Q(v)+o(1)$ as $\varepsilon\to 0$
and the desired property of $\Phi$
is reduced to the already established property of $Q$.
\end{proof}

\medskip

The minimization problem (\ref{costagain}) can be rephrased into a constrained minimization problem in an infinite-dimensional space. For simplicity, consider the case where $M=\Real^n$. Let $(x(\cdot),u(\cdot))$ be an admissible pair of the control system (\ref{driftcontrol}) and let $\varphi:\Real^n\times L^\infty([0,1],\Real^k)\to\Real$ be the function defined by
\[
\varphi(x,u(\cdot))=\int_0^1 L(x(t),u(t))dt.
\]
Let $\Phi:\Real^n\times L^\infty([0,1],\Real^k)\to \Real^n\times \Real^n$ be the map
\[
\Phi(x,u(\cdot))=(x,End_x(u(\cdot))).
\]
Finding the minimum in (\ref{costagain}) is now equivalent to minimizing the function $\varphi$ on the set $\Phi^{-1}(x,y)$.

Due to the above discussion, we consider the following general setting. Consider a function $\varphi:E\to\mathbb R$ on the Banach space $E$ such that $\varphi|_W$ is a $C^2$-mapping for any finite dimensional subspace
$W$ of $E$.  Assume that the function $\varphi$ as well as the first
and second derivatives of the restrictions $\varphi|_W$ are continuous on the bounded subsets of $E$ in the
topology of $H$. Assume that $K$ is a bounded subset of $E$ that is compact in the topology of $H$ and
satisfies the following property:
\[
\varphi(v)=\min\{\varphi(w)|w\in E,\ \Phi(w)=\Phi(v)\}
\]
for any $v$ in the set $K$.

We define a function $\mu$ on $\Phi(K)$ by the formula $\mu(\Phi(v))=\varphi(v),\ v\in K$.

\medskip

\begin{lem}\label{Lipschitz}
If $\mathrm{Ind}_v\Phi\ge 2$ for any $v\in K$, then $\mu$ is locally Lipschitz.
\end{lem}

\medskip

\begin{proof}
Given $v$ in $K$, there exists a finite dimensional
subspace $W$ of the Banach space $E$ such that
$\mathrm{Ind}_v\left(\Phi\bigr|_W\right)\ge 2$. Then
$\mathrm{Ind}_v\left(\Phi\bigr|_{W\cap\ker D_v\varphi}\right)\ge
0$. Hence $\Phi\bigr|_{W\cap\ker D_v\varphi}$ is 2-solid at $v$
and
$$
\Phi\left(\mathbf B_v(\varepsilon)\cap W\cap\ker
D_v\varphi\right)\supset B_{\Phi(v)}(c\varepsilon^2)
$$
for some c and any sufficiently small $\varepsilon$.

Let $x=\Phi(v)$ and $|x-y|=c\varepsilon^2$, then $y=\Phi(w)$ for some
$w\in \mathbf B_v(\varepsilon)\cap W\cap\ker D_v\varphi$. We have:
$$
\mu(y)-\mu(x)\le \varphi(w)-\mu(x)=\varphi(w)-\varphi(v)\le
c'|w-v|^2\le c'\varepsilon^2.
$$
Moreover, the compactness of $K$ allows to chose unique $c,c'$ and
the bound for $\varepsilon$ for all $v\in K$. In particular, we
can exchange $x$ and $y$ in the last inequality. Hence
$|\mu(y)-\mu(x)|\le\frac{c'}c|y-x|.$
\end{proof}

\begin{proof}[Proof of Theorem \ref{lip}]

We perform the proof only in the case $M=\mathbb R^n$
in order to simplify the language. Generalization to any manifold
is straightforward. We set
$$
E=\mathbb R^n\times L^\infty([0,T],\Real^k),\
H=\mathbb R^n\times L^2([0,T],\Real^k),
$$
$$
 \Phi(x,u(\cdot))=(x,End_x(u(\cdot))),\
\varphi(x,u(\cdot))=\int_0^1L(x(t),u(t))\,dt
$$
and apply the above results.

First of all, $\mathrm{Ind}_{(x,u(\cdot))}\Phi=\mathrm{Ind}_{u(\cdot)}
End_x=+\infty$ for all $(x,u(\cdot))$ since our system does
not admit sharp controls. Lemma \ref{2solid} implies that $\Phi$ is 2-solid
and $\mathcal D=\Phi(E)$ is open.

Now let $\mathcal B$ be a ball in $E$ equipped with the weak
topology of $H$. The {\it endpoint mapping} $\Phi$ is continuous
as a mapping from $\mathcal B$ to $\mathbb R^{2n}$. Strict
convexity of $L$ implies that there is some constant $c>0$ such that
$$
\varphi(x_n,u_n(\cdot))-\varphi(x,u(\cdot))\ge c\|u_n(\cdot)-u(\cdot)\|^2_{L^2}+o(1)
$$
as $x_n\to x,\ u_n(\cdot)\rightharpoonup u(\cdot)$ and $\ (x_n,u_n(\cdot))\in\mathcal B$. Therefore,
$\lim\limits_{n\to\infty}\varphi(x_n,u_n(\cdot))\ge\varphi(x,u(\cdot))$ and
$\lim\limits_{n\to\infty}\varphi(x_n,u_n(\cdot))=\varphi(x,u(\cdot))$ if and
only if $(x_n,u_n(\cdot))$ converges to $(x,u(\cdot))$ in the strong
topology of $H$.

Assume that $\varphi(x_n,u_n(\cdot))=\mu(\Phi(x_n,u_n(\cdot)))$ for all $n$.
Inequality $\varphi(x,u(\cdot))<\lim\limits_{n\to\infty}\varphi(x_n,u_n(\cdot))$
would imply that

$$\mu(\Phi(x,u(\cdot)))<\lim\limits_{n\to\infty}\mu(\Phi(x_n,u_n(\cdot))).$$

On the other hand, the openness of the map $\Phi$ implies that
$$\mu(\Phi(x,u(\cdot)))\ge\lim\limits_{n\to\infty}\mu(\Phi(x_n,u_n(\cdot))).$$
Hence $\lim\limits_{n\to\infty}\varphi(x_n,u_n(\cdot))=\varphi(x,u(\cdot))$ and
$(x_n,u_n(\cdot))$ converges to $(x,u(\cdot))$ in the strong topology of $H$.

Let $C$ be a compact subset of $\mathcal D$ and
$$
K=\left\{(x,u(\cdot))\in E: \Phi(x,u(\cdot))\in C,\
\varphi(x,u(\cdot))=\mu(\Phi(x,u(\cdot)))\right\}.
$$
Then $K$ is contained in some ball $\mathcal B$. Recall that
$\mathcal B$ is equipped with the weak topology; it is compact.
Now calculations of previous 2 paragraphs imply compactness of
$K$ in the strong topology of $H$. Finally, we derive the Lipschitz
property of $\mu|_C$ from Lemma \ref{Lipschitz}.
\end{proof}

\bigskip

\section{Applications: Mass Transportation on Subriemannian Manifolds}

In this section, we will apply the results in the previous sections
to some subriemannian manifolds. First, let us recall some basic
definitions.

Let $\Delta$ and $\Delta'$ be two (possibly singular) distributions on a manifold $M$. Define the distribution
$[\Delta,\Delta']$ by
$$
[\Delta,\Delta']_x=\text{span}\{[v,w](x)|v \text{ is a section
of }\Delta, w \text{ is a section of }\Delta'\}.
$$
Define inductively the following distributions:
$[\Delta,\Delta]=\Delta^2$ and $\Delta^k=[\Delta,\Delta^{k-1}]$. A
distribution $\Delta$ is called $k$-generating if $\Delta^k=TM$
and the smallest such $k$ is called the degree of nonholonomy.
Also, the distribution is called bracket generating if it is
$k$-generating for some $k$.

If $\Delta$ is a bracket generating distribution, then it defines a flag of distribution by
\[
\Delta\subset\Delta^2\subset...\subset TM.
\]
The growth vector of the distribution $\Delta$ at the point $x$ is defined by
$(\dim\Delta_x,\dim\Delta^2_x,...,\dim T_xM)$. The distribution $\Delta$ is called
regular if the growth vector is the same for all $x$. Let $x(\cdot):[a,b]\to M$ be an
admissible curve, that is a Lipschitz curve almost everywhere tangent to $\Delta$.
The following classical result on bracket generating distributions is the starting point of subriemannian geometry.

\medskip

\begin{thm}(Chow and Rashevskii)
Given any two points $x$ and $y$ on the manifold $M$ with a bracket
generating distribution, there exists an admissible curve joining
the two points.
\end{thm}

\medskip

Using Chow-Rashevskii Theorem, we can define the subriemannian distance $d$. Let $<,>$
be a fibre inner product on the distribution $\Delta$, called subriemannian metric.
The length of an admissible curve $x(\cdot)$ is defined in the usual
way: $\text{length}(x(\cdot))=\int_a^b\sqrt{<\dot x(t),\dot x(t)>}\,dt$. The subriemannian
distance $d(x,y)$ between two points $x$ and $y$ is defined by the infimum of the length
of all admissible curves joining $x$ and $y$.
There is a quantitative version of Chow-Rashevskii Theorem, called Ball-Box Theorem,
which gives H\"{o}lder continuity of the subriemannian distance. See \cite{Mo2} for detail.

\medskip

\begin{cor}\label{lip2}
Let $d_S$ be the metric of a complete subriemannian space with distribution
$\Delta$. Function $d_S^2$ is locally Lipschitz if and only if the distribution is
2-generating.
\end{cor}

\medskip

\begin{proof}
The systems with 2-generating distributions do not admit sharp paths because
these systems are not compatible with the Goh condition. On the other hand, constant
paths (points)
are sharp minimizers in the case of distributions whose nonholonomy degree is greater
than 2 and
the ball-box theorem implies that $d^2$ is not locally Lipschitz at the diagonal in
this case.
\end{proof}

Combining Corollary \ref{lip2} with Theorem \ref{main}, we prove the existence and uniqueness of optimal map for subriemannian manifold with 2-generating distribution.

\begin{thm}\label{maintransport}
Let $M$ be a complete subriemannian manifold defined by a 2-generating distribution, then there exists a unique (up to $\mu$-measure zero) optimal map to the Monge's problem with the cost $c$ given by $c=d_S^2$. Here $d_S$ is the subriemannian distance of $M$.
\end{thm}

\medskip

\begin{rem}
The locally Lipschitz property of the distance $d$ out of the diagonal is guaranteed
for much bigger class of distribution. In particular, it is proved in \cite{AgGa} that
generic distribution of rank $>2$ does not admit non-constant sharp trajectories. In the
class of Carnot groups, the following estimates are valid: Generic $n$-dimensional
Carnot group with rank $k$ distribution does not admit nonconstant sharp trajectories
if $n\le(k-1)k+1$ and has nonconstant sharp length minimizing trajectories if
$n\ge(k-1)(\frac{k^2}3+\frac{5k}6+1)$. Recall that a simply-connected Lie group
endowed with a left-invariant distribution $V_1$ is a Carnot group if the Lie algebra
$\mathfrak g$ is a graded nilpotent Lie algebra such that it is Lie generated by the
block with lowest grading
(i.e. $\mathfrak g=V_1\oplus V_2\oplus...\oplus V_k$, $[V_i,V_j]=V_{i+j}$,
$V_i=0 \text{ if $i>k$}$ and $V_1$ Lie-generates $\mathfrak g$).

Clearly, if the cost is locally Lipschitz out of the diagonal,
then the statement of Theorem~4.1 remains valid with the extra
assumption that the supports of the initial $\mu$ and the final
measures $\nu$ are disjoint: $supp(\mu)\cap supp(\nu)=\emptyset$.
\end{rem}

\bigskip

\section{Normal minimizers and Property of Optimal Map with Continuous Optimal Control Cost}

According to Theorem \ref{lip}, it remains to study the case where sharp controls exist.
In this section, we will prove a property of optimal map when the cost is continuous.
Normal minimizers will play a very important role.

We continue to study optimal control problem (20), (21). As we
already mentioned, strictly abnormal minimizers must be sharp. In
addition, if $X_0=0$, then optimal control cost is continuous.
According to the discussion at the end of the previous section, we
expect strictly abnormal minimizers mainly for generic rank 2
distributions on the manifold of dimension greater than 3 and for
generic Carnot groups of big enough corank. In these situations,
strictly abnormal minimizers are indeed unavoidable.

The existence of strictly abnormal minimizers for subriemannian manifolds is first done
in \cite{Mo1}. In \cite{Su} and \cite{LiSu}, it is shown that there are many strictly
abnormal minimizers in general for subriemannian manifolds. (See, for instance,
Theorem \ref{LiuSus} below.) Finally, a general theory on abnormal minimizers for
rank 2 distributions is developed in \cite{AgSa2}. See \cite{Mo2} for a detail account
on the history and references on abnormal minimizers.

Here is a sample result in \cite{Su} which is of interest to us.

\begin{thm}\label{LiuSus}(Liu and Sussman)
Let $M$ be a 4-dimensional manifold with a rank 2 regular bracket generating
distribution $\Delta$ and subriemannian metric $<,>$. Let $X_1$ and $X_2$ be two global sections of $\Delta$ such that
\begin{enumerate}
\item $X_1$ and $X_2$ are everywhere orthonormal,
\item $X_1$, $X_2$, $[X_1,X_2]$ and $[X_2,[X_1,X_2]]$ are everywhere linearly dependent,
\item $X_2$, $[X_1,X_2]$ and $[X_2,[X_1,X_2]]$ are everywhere linearly independent.
\end{enumerate}
Then any short enough segments of the integral curves of the vector field $X_2$ are strictly abnormal minimizers.
\end{thm}

We call a local flow a strictly abnormal flow if the corresponding trajectories are all strictly abnormal minimizers. An interesting question is whether time-1 map of an abnormal flow is an optimal map. The following theorem shows that this is not the case for any reasonable initial measure and continuous cost.

\medskip

\begin{thm}\label{NormOpt}
Assume that the cost $c$ in (\ref{optimal}) is continuous, bounded below and the support of the measure $\mu$ is equal to the closure of its interior. If $\varphi:M\to M$ is a continuous map such that $(id\times\varphi)_*\mu$ achieves the infimum in Problem \ref{KanRe}, then $x$ and $\varphi(x)$ are connected by a normal minimizer on a dense set of $x$ in the support of $\mu$.
\end{thm}

\medskip

\begin{proof}
By Theorem \ref{dualexist}, there exists a function $f:M\to\Real\cup\{-\infty\}$ such that $f$ and its $c_1$-transform achieve the supremum in Problem \ref{dual}. Moreover, by Theorem \ref{dualregular}, the functions $f$ and $f^{c_1}$ are upper semicontinuous. By Theorem \ref{existjoint},
\begin{equation}\label{NormOpt1}
f(x)+f^{c_1}(\varphi(x))=c(x,\varphi(x))
\end{equation}
for $\mu$-almost all $x$. By upper semicontinuity of $f$ and $f^{c_1}$,
\[
f(x)+f^{c_1}(\varphi(x))\geq c(x,\varphi(x)).
\]
But $f(x)+f^c(y)\leq c(x,y)$ for any $x,y$ in the manifold $M$. So, (\ref{NormOpt1}) holds for all $x$ in the support $U$ of $\mu$. Therefore, $x$ achieves the infimum $f^{c_1}(\phi(x))=\inf_{z\in M}[c(z,\phi(x))-f^{c_1}(z)]$ for all $x$ in the support of $\mu$. Moreover, using (\ref{NormOpt1}), it is easy to see that the function $f$ is continuous on $U$. In particular, it is subdifferentiable on a dense set of $U$.
By Proposition \ref{dualmin} and Theorem \ref{appPMP}, $x$ and $\varphi(x)$ is connected by a normal minimizer if $f$ is subdifferentiable at $x$. This proves the theorem.
\end{proof}

\bigskip

\section{Optimal Maps with Abnormal Minimizers}

In this section, we describe an important class of control systems
which admit smooth optimal maps built essentially from abnormal
minimizers. Recall that abnormal minimizers are singular
trajectories of the control system whose definition does not depend on the
Lagrangian.

Let $\rho:M\stackrel{G}{\To}N$ be a smooth principal bundle where
the structural group $G$ is a connected Abelian Lie group. Let
$X_1,\ldots,X_k$ be the vertical vector fields which generate
the action of $G$. Consider the following control system

\begin{equation}\label{controlagainagain}
\dot x(t)=X_0(x(t))+\sum\limits_{i=1}^ku_i(t)X_i(x(t)),
\end{equation}
where $X_0$ is a smooth vector field on $M$, and the re-scaled systems
\begin{equation}\label{controlrescaled}
\dot x(t)=\varepsilon X_0(x(t))+\sum\limits_{i=1}^ku_i(t)X_i(x(t))
\end{equation}
for $\varepsilon>0$.

We define the Hamiltonian $H:T^*N\to\Real$ by
\begin{equation}\label{abHam}
H(p_x)=\max\{p_x(d\rho(X_0(y))|y\in\rho^{-1}(x)\}
\end{equation}
where $p_x$ is a covector in $T^*N$. We assume that the maximum
above is achieved for any $p$ in  $T^*N$ and is finite.

Typical example is the Hopf bundle $
\phi:\mathrm{SU(2)}\stackrel{S^1}{\longrightarrow}S^2 $ and a
left-invariant vector field $F_0$. Then $H(p)=\alpha|p|$, where
$\alpha$ is a constant and $|p|$ is the length of the covector $p$
with respect to the standard (constant curvature) Riemannian
structure on the sphere. (See \cite[Section 22.2]{AgSa1})

Consider the following control system on $N$ with admissible pair $y(\cdot)$ contained in the $G$-bundle $\rho:M\stackrel{G}{\To}N$ and admissible trajectory $x(t)=\rho(y(t))$ (See Remark \ref{generalcontrol}):
\begin{equation}\label{reducedcontrol}
\dot x(t)=d\rho(X_0(y(t))).
\end{equation}
The function $H$ in (\ref{abHam}) is the Hamiltonian of the time-optimal problem of the control system (\ref{reducedcontrol}). (Recall that the time optimal problem is the following minimization problem: Fix two points $x_0$ and $x_1$ in $N$ and minimize the time $t_1$ among all admissible trajectories $x(\cdot)$ of the control system (\ref{reducedcontrol}) such that $x(t_0)=x_0$ and $x(t_1)=x_1$.)

System (\ref{reducedcontrol}) is the reduced system
associated to system (\ref{controlagainagain}) according to the reduction procedure
described in \cite[Chapter 22]{AgSa1}. In particular, $\rho$ transforms any
admissible trajectory of system (\ref{controlagainagain}) to the admissible trajectory of
system (\ref{reducedcontrol}). Also, the smooth extremal trajectories of the time-optimal
problem for system (\ref{reducedcontrol}) are images under the map $\rho$ of singular
trajectories of system (\ref{controlagainagain}).

For any $\varepsilon>0$ and any $C^2$ smooth function $a:N\to\Real$, we introduce the map
$$
\Phi_a^\varepsilon:N\to N,\quad \Phi_a^\varepsilon(x)=\pi(e^{\varepsilon\vec H}(d_xa)),\ x\in N,
$$
where $\pi:T^*N\to N$ is the standard projection and $t\mapsto e^{t\vec H}$ is the Hamiltonian flow of $H$. Set
$$
\mathcal D=\{p\in T^*N : H(p)>0,\ H \ \mathrm{is\ of\ class}\ C^2\
\mathrm{at}\ p\}.
$$

Assume that $\Phi_a^\e$ pushes the measure $\mu'$ forward to another measure $\nu'$ on $N$. Consider some ``lifts'' $\mu$ and $\nu$ of the measures $\mu'$ and $\nu'$: $\rho_*\mu=\mu', \rho_*\nu=\nu'$. Let $\Psi:M\To M$ be an optimal map pushing forward $\mu$ to $\nu$, then the following theorem says that $\Psi$ is a covering of $\Phi^\varepsilon_a$: $\rho\circ\Psi=\Phi^\varepsilon_a\circ\rho.$ By the discussion above, we see that $x$ and $\Psi(x)$ are connected by singular trajectories as claimed.

\begin{thm}
Let $K$ be a compact subset of $N$ and $a\in C^2(N)$. Assume that $da|_K\subset\mathcal D$. Let $\mu$ and $\nu$ be Borel probability measures such that $\mathrm{supp}(\rho_*(\mu))\subset K$. Then, for any sufficiently small
$\varepsilon>0$ and any optimal Borel map $\Psi:M\to M$ of the control system (\ref{controlrescaled}) with any
Lagrangian $L$, the following is true whenever $\rho_*(\nu)={\Phi^\varepsilon_a}_*(\rho_*(\mu))$:
\[
\rho\circ\Psi=\Phi^\varepsilon_a\circ\rho.
\]

\end{thm}

\begin{proof}
We start from the following.
\medskip\noindent

\begin{defn}
We say that a Borel map $Q:K\to N$ is $\varepsilon$-admissible for system (\ref{controlrescaled}) if there exists a Borel map $\varphi:K\to L^\infty([0,\varepsilon],G)$ such that
$$
Q(x_0)=x\left(\varepsilon;\varphi(x_0)(\cdot)\right),\quad \forall
x_0\in K,
$$
where $t\mapsto x\left(t;\varphi(x_0)(\cdot)\right)$ is an admissible trajectory of the reduced control system (\ref{reducedcontrol}) with initial condition $x\left(0;\varphi(x_0)(\cdot)\right)=x_0$.
\end{defn}

\medskip

We are going to prove that $\Phi^\varepsilon_a$ is an admissible map, unique up to
a $\rho_*(\mu)$-measure zero set, which transforms
$\rho_*(\mu)$ into $\rho_*(\nu)$. This fact implies the
statement of the theorem.

Inequality $H(d_xa)>0$ implies that $d\pi(\vec H(d_xa))$ is
transversal to the level hypersurface of $a$ through $x$. Hence the
map $\Phi^\varepsilon_a$ is invertible on a neighborhood of $K$ for
any sufficiently small $\varepsilon$. Moreover, the curve
$t\mapsto\Phi_a^t(y),\ 0\le t\le\varepsilon$, is a unique admissible
trajectory of system \ref{reducedcontrol} which starts at the hypersurface
$a^{-1}(a(x))$ and arrives at the point $\Phi^\varepsilon_a(x)$ at
time moment not greater than $\varepsilon$. The last fact is proved
by a simple adaptation of the standard sufficient optimality
condition (see \cite[Chapter 17]{AgSa1}).

Now we set
$$
a_\varepsilon(x)=a\left((\Phi^\varepsilon_a)^{-1}(x)\right)+\varepsilon,
$$
then $a_\varepsilon$ is a smooth function defined on a
neighborhood of $K$.

Optimality property of $\Phi^\varepsilon_a$ implies that
$$
a_\varepsilon(Q(x))\le
a_\varepsilon\left(\Phi^\varepsilon_a(x)\right)
$$
for any $\varepsilon$-admissible  map $Q$ and any $x\in K$, and
the inequality is strict at any point $x$ where
$Q(y)\ne\Phi^\varepsilon_a(x)$. In particular, if
$$
\rho_*(\mu)\left(\{x\in K :
Q(x)\ne\Phi^\varepsilon_a(x)\}\right)>0,
$$
then
$$
\int a_\varepsilon\, d(Q_*(\rho_*(\mu)))=\int a_\varepsilon\circ
Q\,d(\rho_*(\mu))<
$$
$$
\int
a_\varepsilon\circ\Phi^\varepsilon_a\,d(\rho_*(\mu)) =\int
a_\varepsilon\,d(\rho_*(\nu)).
$$
Hence $Q_*(\rho_*(\mu))\ne \rho_*(\nu).$
\end{proof}

\bigskip

\section{Example: the Grushin plane}

Grushin plane is the subriemannian space with base space $\Real^2$ and a singular
distribution defined by the span of the following vectors
$\{\partial_{x_1},x_1\partial_{x_2}\}$ in each tangent space. In other word, the fibre
of this distribution is the whole tangent space of $\Real^2$ if $x_1\neq 0$ and it is
spanned by $\partial_{x_1}$ otherwise. We define a subriemannian metric by declaring
that the two vector fields above are orthonormal. The control system is given by

$$ \dot x_1=u_1,\quad \dot x_2=u_2x_1.$$

The subriemannian distance $d$ is given by
$d(x,y)=\inf\limits_{\C_x^y}\int_0^1\sqrt{u_1^2+u_2^2}\,dt$. In this section, we consider the optimal transport problem with cost $c$ given by $c=d^2$.

There is no abnormal minimizer for this problem, so we consider its Hamiltonian $H$ given by
\[
H(x_1,x_2,p_1,p_2)=\frac{1}{2}(p_1^2+x_1^2p_2^2).
\]

The corresponding Hamiltonian equation is
\[
\dot x_1=p_1,\quad \dot x_2=x_1^2p_2, \quad \dot p_1=-x_1p_2^2, \quad \dot p_2=0.
\]

For simplicity, we consider the case $x_1(0)=0=x_2(0)$. And we let $p_1(0)=a$ and
$p_2(0)=b$.
In this case, the solutions give geodesics emanating from a point $(0,\delta)$ on the
$y$-axis. They are parameterized by $(a,b)$ and are given by
\begin{equation}\label{geo1}
x_1(t)=\frac{a}{b}\sin(bt),\quad x_2(t)=\frac{a^2}{4b^2}(2bt-\sin(2bt))+\delta
\end{equation}
if $b\neq 0$ and given by
\begin{equation}\label{geo2}
x_1(t)=at,\quad x_2(t)=\delta
\end{equation}
if $b=0$. A geodesic is length minimizing if and only if $-\pi/b\leq t\leq\pi/b$.

\begin{figure}
\includegraphics[angle=270,scale=0.6]{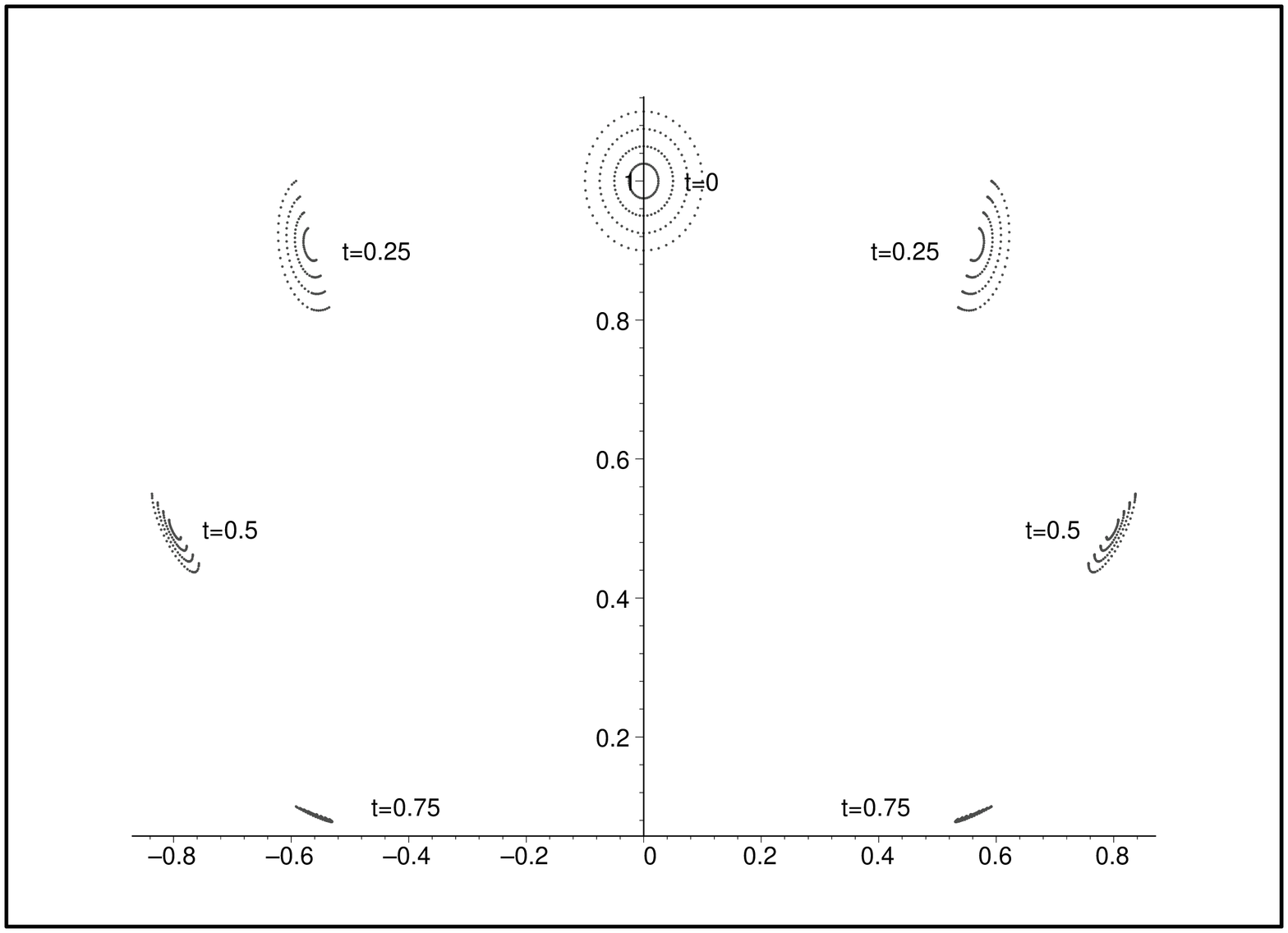}
\includegraphics[angle=270,scale=0.6]{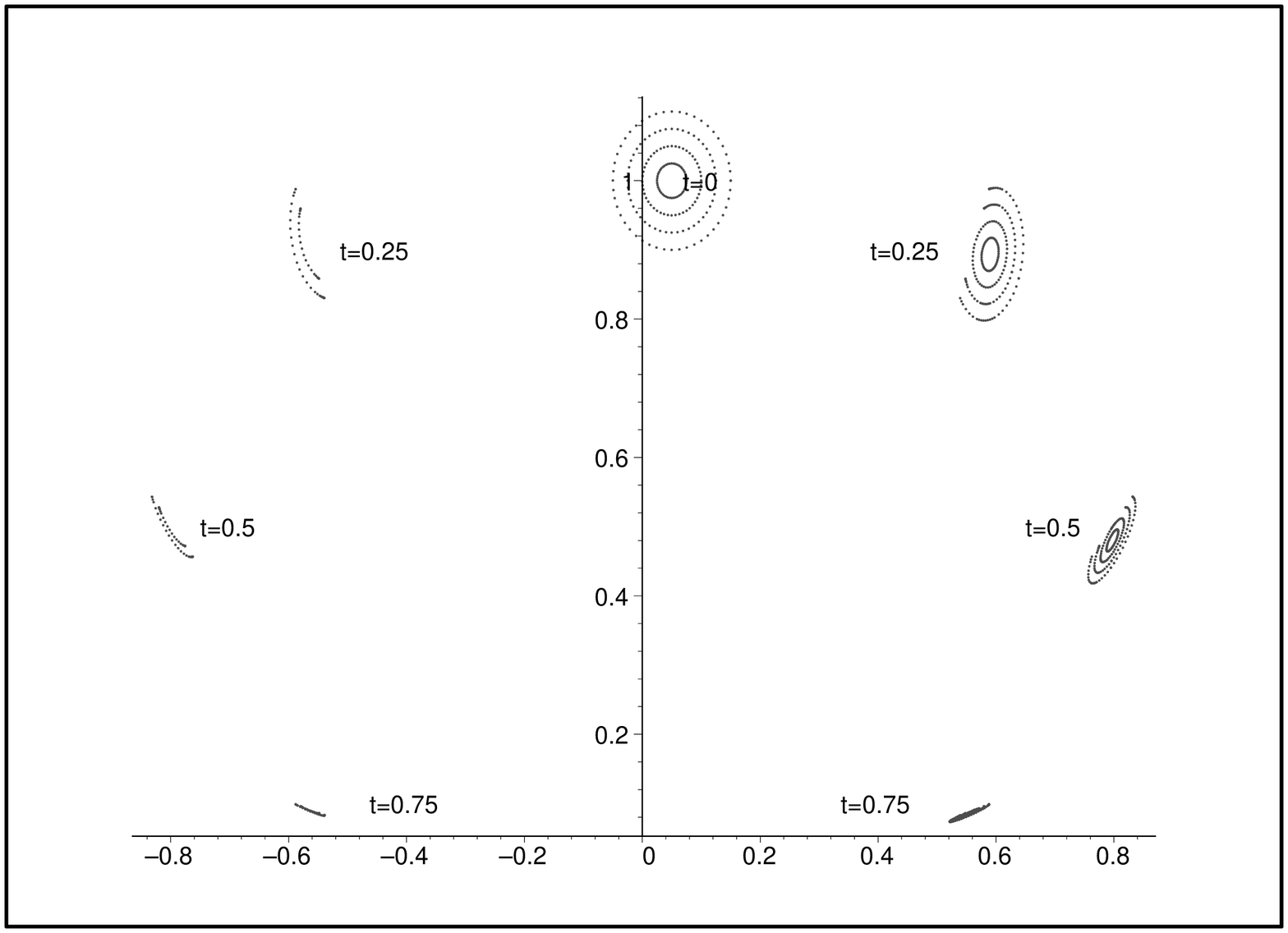}
\end{figure}

\begin{figure}
\includegraphics[angle=270,scale=0.6]{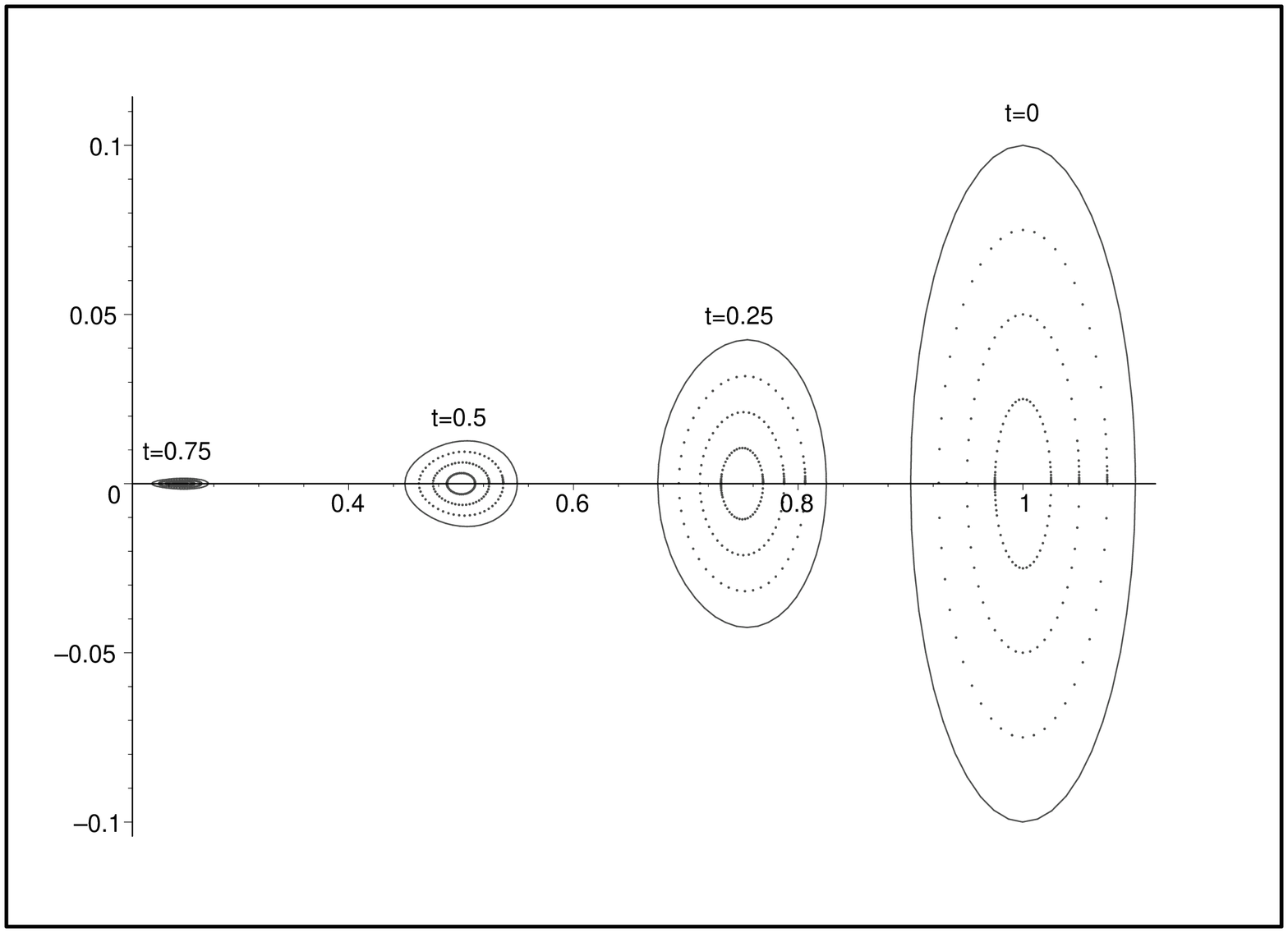}
\includegraphics[angle=270,scale=0.6]{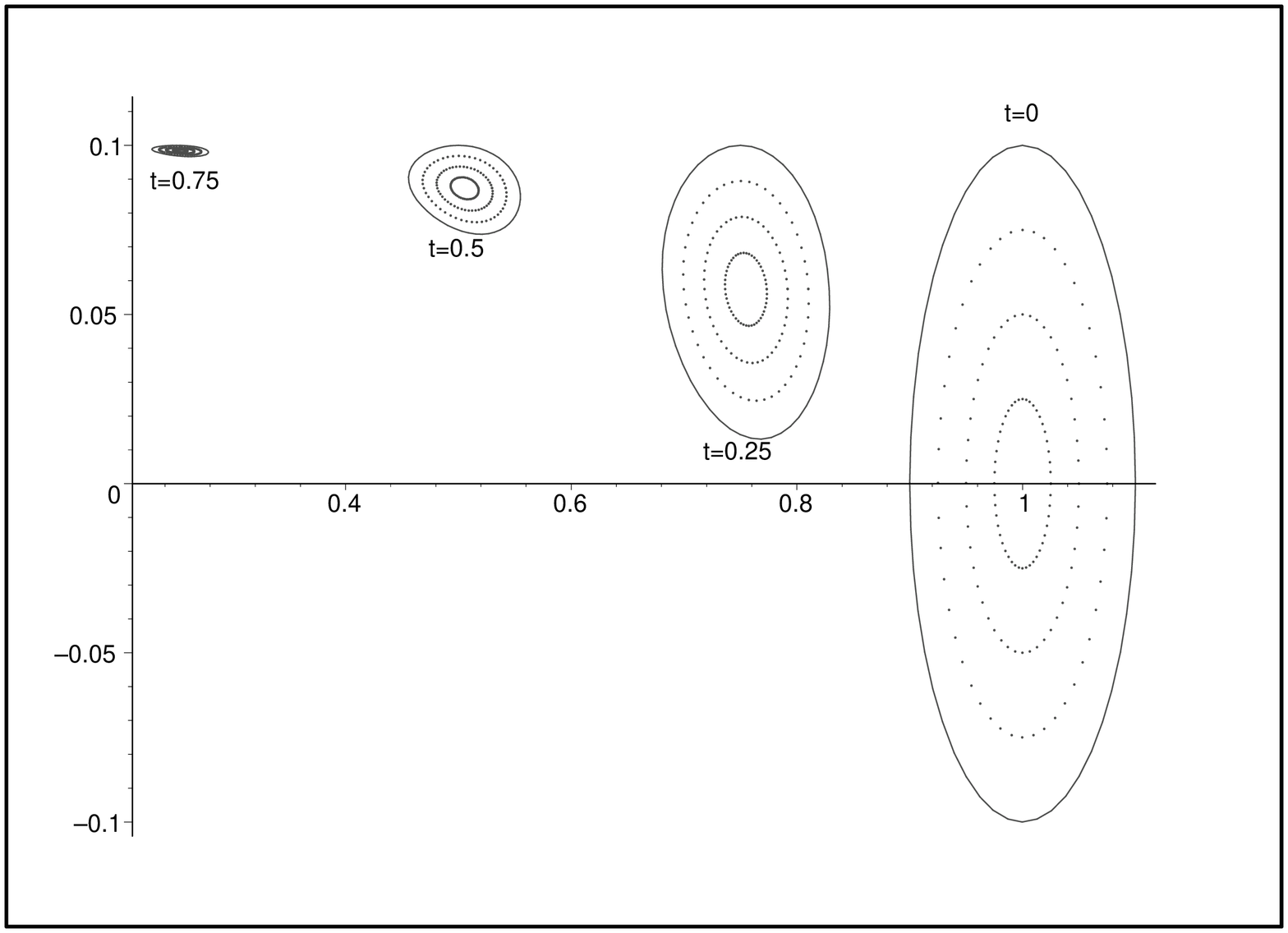}
\caption{Some displacement interpolations}
\end{figure}

\epsfysize=10 cm
\begin{figure}[htb]
\center{
\leavevmode
\rotatebox{270}{\epsfbox{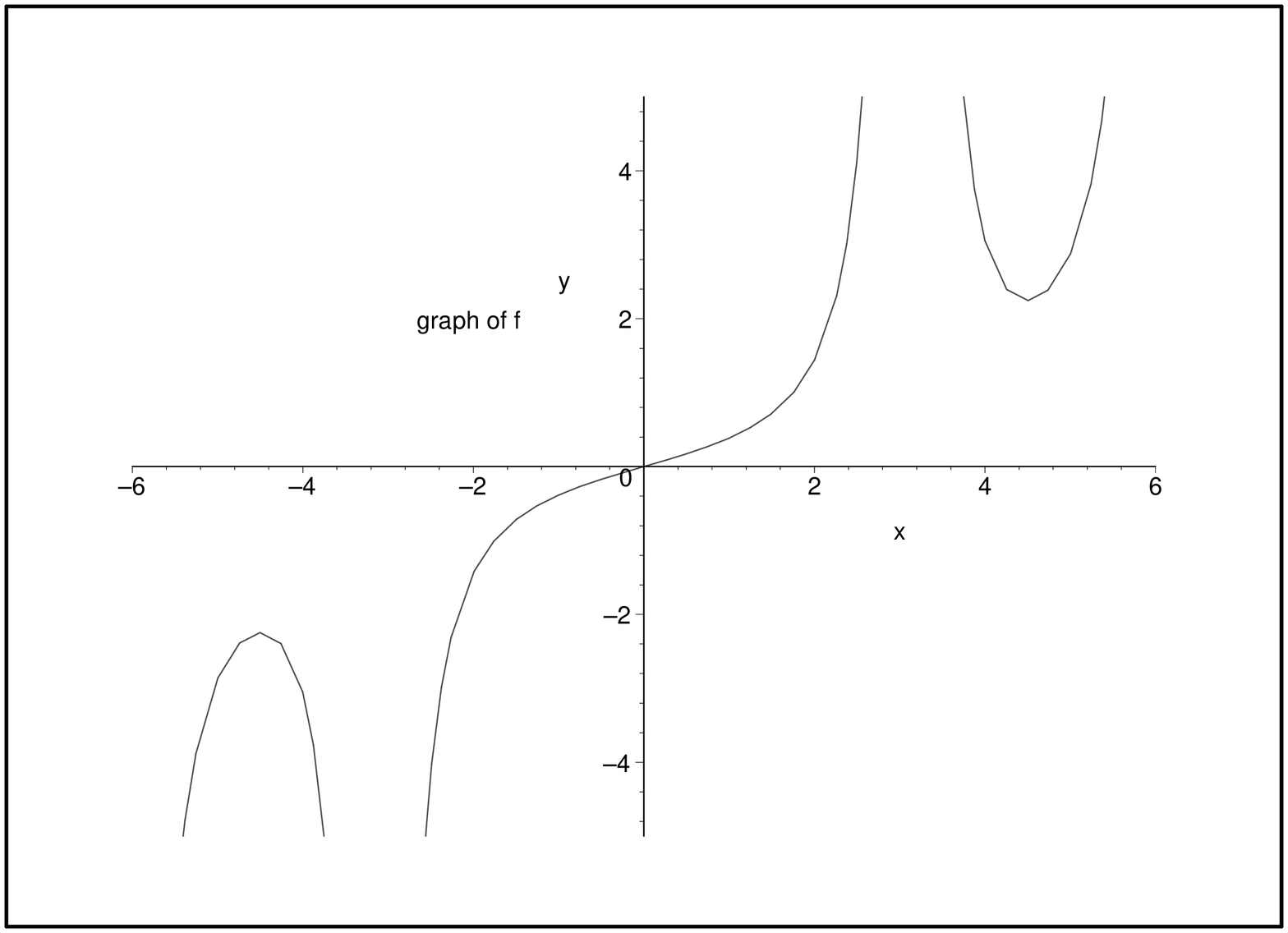}}
\caption{Graph of the function $f$}
}
\end{figure}

Next, we consider the mass transport problem. Let $d$ be the subriemannian distance of the Grushin plane and consider Problem \ref{optimal}with cost $c$ given by square of the subriemannian distance $d^2$. We also specialize to the case where the target measure $\nu$ is equal to the delta mass supported at the origin. In this case, the optimal map is clearly given by the constant map $x\mapsto (0,0)$. We are interested in the displacement interpolation corresponding to this optimal map. Recall that displacement interpolation is the one parameter family of maps $\phi_t$ such that $\phi_t$ is the optimal map with the cost $c_t$ given by the following:
\[
c_t(x,y)=\inf\int_0^tL(x(s),u(s))\,ds
\]
where the infimum ranges over all admissible pairs $(x(\cdot),u(\cdot))$ of the control system (\ref{control}) with initial condition $x(0)=x$ and final condition $x(t)=y$. It is easy to see that if $\phi_1=\pi(e^{\overrightarrow{H}}(-df))$ as in Theorem \ref{main}, then the displacement interpolation $\phi_t$ is given by $\pi(e^{t\overrightarrow{H}}(-df))$. Moreover, the displacement interpolation is related to the Hamilton-Jacobi equation via the method of characteristics. See \cite{BeBu} and \cite{CaSi} for details.

To do this, we first evaluate the equations (\ref{geo1}) and (\ref{geo2}) at $t=1$. Then we solve $a$ and $b$ in terms of $x_1(1)$ and $x_2(1)$. If $f:(-\pi,\pi)\to\Real$ is the function defined by $f(b)=\frac{2b-\sin(2b)}{4\sin^2(b)}$, then $f$ is invertible. A computation shows that
\[
a=\frac{f^{-1}\Big(\frac{x_2(1)-\delta}{x_1(1)^2}\Big)x_1(1)} {\sin\Big(f^{-1}\Big(\frac{x_2(1)-\delta}{x_1(1)^2}\Big)\Big)},\quad b=f^{-1}\Big(\frac{x_2(1)-\delta}{x_1(1)^2}\Big).
\]

Therefore, the displacement interpolation is given by
\[
\varphi_t(x_1,x_2)=\Big(\frac{a}{b}\sin(b(1-t)) ,\frac{a^2}{4b^2}(2tb-\sin(2(1-t)b)+\delta\Big),
\]
where $a=a(x_1,x_2)$ and $b=b(x_1,x_2)$ are given by
\[
a(x_1,x_2)=\frac{f^{-1}\Big(\frac{x_2-\delta}{x_1^2}\Big)x_1} {\sin\Big(f^{-1}\Big(\frac{x_2-\delta}{x_1^2}\Big)\Big)},\quad b(x_1,x_2)=f^{-1}\Big(\frac{x_2-\delta}{x_1^2}\Big).
\]

\section{Appendix}

This appendix is devoted to the prove of Theorem \ref{PMPB}. The first step is to reduce the problem into a simpler one. Recall that the Bolza problem is the following minimization problem:

\[
\inf_{(x(\cdot),u(\cdot))\in\C_{x_0}}\int_0^1L(x(s),u(s))\,ds-f(x(1))
\]
where the infimum is taken over all admissible pair $(x(\cdot),u(\cdot))$ satisfying the control system
\[
\dot x(s)=F(x(s),u(s))
\]
and initial condition $x(0)=x_0$.

Let $\br x=(x,z)$ be a point in the product manifold $M\times\Real$
and consider the following extended control system on it:

\begin{equation}\label{extendcontrol}
\dot{\br x}=\br F(\br x,u):=(F(x,u),L(x,u)).
\end{equation}

Note that $\br x(\cdot)=(x(\cdot),z(\cdot))$ satisfies this extended system and initial
condition $\br x(0)=(x_0,0)$ if and only if $x(\cdot)$ satisfies the
original control system in the Bolza problem with the initial condition
$x(0)=x_0$ and $z(t)=\int_0^tL(q(s),u(s))\,ds$. Therefore, Problem
\ref{Bolza} is equivalent to the following problem.

\begin{prob}\label{extendmin}
\begin{equation}
\inf_{(\br x(\cdot),u(\cdot))\in\C_{(x_0,0)}}\left(z(1)-f(x(1))\right),
\end{equation}
where the infimum is taken over all admissible pair satisfying the extended control system (\ref{extendcontrol}).
\end{prob}

Problem \ref{extendmin} is an example of the Mayer problem. Let $g:N\to\Real$ be a function on
the manifold $N$ and the Mayer problem is the following minimization problem:

\begin{prob}\label{Mayer}
\[
\inf_{\C_{\br x_0}} g(\br x(1))
\]
where the infimum is taken over all admissible pair $(\br x(\cdot),u(\cdot))$ satisfying the control system
\[
\dot{\br x}=\br F(\br x,u)
\]
on $N$ and initial condition $\br x(0)=\br x_0$.
\end{prob}

For each point $u$ in the control set $U$, define the corresponding
Hamiltonian function $\br H_u:T^*N\to\Real$ by
\[
\br H_u(p_{\br x})=p_{\br x}(\br F(\br x,u)).
\]

\medskip

\begin{thm}\label{PMP}(Pontryagin Maximum Principle for Mayer Problem)

Let $(\tl {\br x}(\cdot),\tl u(\cdot))$ be an admissible pair which achieve the infimum
in Problem \ref{Mayer}. Assume that the function $g$ in Problem
\ref{Mayer} is super-differentiable at the point $\tl {\br x}(1)$ and let
$\br \alpha$ be in the super-differential $d^+g_{\tl {\br x}(1)}$ of $g$. Then
there exists a Lipschitz path $\tl p(\cdot):[0,1]\to T^*N$ which satisfies
the following for almost all time $t$ in the interval $[0,1]$:

\begin{equation}
\left\{%
\begin{array}{ll}
    \pi(\tl p(t))=\tl {\br x}(t),\\
    \tl p(1)=\br \alpha,\\
    \dot{\tl p}(t)=\overrightarrow {\br H}_{\tl u(t)}(\tl p(t)),\\
    \br H_{\tl u(t)}(\tl p(t))=\min\limits_{u\in U}\br H_u(\tl p(t))
\end{array}%
\right.
\end{equation}

\end{thm}

\medskip

\begin{proof}
Fix a point $v$ in the control set and a number $\tau$ in the
interval $[0,1]$. For each small positive number $\e>0$, let $\ue$
be the admissible control defined by
\[\ue(t)=\left\{
\begin{array}{ll}
    \tl u (t), & \hbox{if $t\notin[\tau-\e,\tau]$;} \\
    v, & \hbox{if $t\in[\tau-\e,\tau]$.} \\
\end{array}
\right.\]

Since the optimal control $\tl u$ is locally bounded, the new
control $\ue$ defined above is also locally bounded. Let
$\Pe{t_0}{t_1}:N\to N$ be the time-dependent local flow of the
following ordinary differential equation
\[
\dot {\br x}(t)=\br F(\br x(t),\ue(t)).
\]
Here, $\Pe{0}{t}(\br x)$ denotes the image of the point $\br x$ in the
manifold $N$ under the local flow $\Pe{0}{t}$ at time $t$. It has
the property that $\Pe{t_2}{t_3}\circ\Pe{t_1}{t_2}=\Pe{t_1}{t_3}$.
Also, recall that $\Pe{t_0}{t_1}$ depends smoothly on the space
variable, Lipschitz with respect to the time variables.

Since $\tl {\br x}(1)=\Pz{0}{1}(\br x_0)$ and the function $g$ is minimizing
at $\tl {\br x}(1)$, the following is true for all $\e>0$:

\begin{equation}\label{PMP1}
g(\Pe01(\br x_0))\geq g(\Pz01(\br x_0)).
\end{equation}

Let $\br\alpha$ be a point in the super-differentiable $d^+g_{\tl{\br
x}(1)}$ at the point $\tl {\br x}(1)$, then there exists a $C^1$ function
$\phi:N\to\Real$ such that $d\phi_{\tl {\br x}(1)}=\br \alpha$ and $g-\phi$
has a local maximum at $\tl {\br x}(1)$. Combining this with (\ref{PMP1}),
we have

$$
g(\Pz01(\br x_0))-\phi(\Pe01(\br x_0))\leq
$$
$$
g(\Pe01(\br x_0))-\phi(\Pe01(\br x_0)) \leq g(\Pz01(\br x_0))-\phi(\Pz01(\br x_0)).
$$

Simplifying this equation, we get

\begin{equation}\label{PMP2}
\frac{\phi(\Pe01(\br x_0))-\phi(\Pz01(\br x_0))}{\e}\geq 0.
\end{equation}

If $R_t$ denotes the flow of the vector field $\br F_v$, then

\begin{equation}\label{PMP3}
\Pe01=\Pz\tau 1\circ R_\e\circ\Pz0{\tau-\e}.
\end{equation}

So, if we assume that $\tau$ is a point of differentiability of the
map $t\mapsto\Pz0t$ which is true for almost all time $\tau$ in the
interval $[0,1]$, then $\Pe01$ is differentiable with respect to
$\e$ at zero. Therefore, we can let $\e$ goes to 0 in (\ref{PMP2})
and obtain

\begin{equation}\label{PMP4}
\br \alpha\left(\ddepZ\Pe01\right)\geq 0.
\end{equation}

If we differentiate equation (\ref{PMP3}) with respect to $\e$ and
set it to zero, it becomes

\[
\ddepZ\Pe01=(\Pz \tau1)_*(\br F_v-\br F_{\tl u(\tau)})\circ\Pz01.
\]

Substitute this equation back into (\ref{PMP4}), we get the
following:

\begin{equation}\label{PMP5}
((\Pz \tau1)^*\br \alpha) (\br F_v(\tl {\br x}(\tau))-\br F_{\tl u(\tau)}(\tl
{\br x}(\tau)))\geq 0.
\end{equation}

Define $\tl p:[0,1]\to T^*N$ by $\tl p(t)=(\Pz t1)^*\br \alpha$, then
the first two assertions of the theorem are clearly satisfied.

The following is well known (See \cite{AgSa1} or \cite{MaRa}).

\begin{lem}\label{pretau}
Let $\theta=pdq$ be the tautological 1-form on the cotangent bundle
of the manifold $N$, then for each diffeomorphism $P:N\to N$, the
pull back map $P^*:T^*N\to T^*N$ on the cotangent bundle of the
manifold preserves the 1-form $\theta$.
\end{lem}

Let $W_t$ be the time-dependent vector field on the cotangent bundle
of the manifold which satisfies

\[
\ddt(\Pz t1)^*=W_t\circ(\Pz t1)^*
\]
for almost all time $t$ in $[0,1]$. If $\LD_V$ denotes the Lie
derivative with respect to a vector field $V$, then, by Lemma
\ref{pretau}, the following is true for almost all time $t$ in
$[0,1]$:

$$
\LD_{W_t}\theta=0.
$$

If $\omega=-d\theta$ is the canonical symplectic 2-form on the
cotangent bundle, then, by using Cartan's formula, we have
\[
i_{W_t}\omega=d(\theta(W_t)).
\]
Therefore, the vector field $W_t$ is a Hamiltonian vector field with
Hamiltonian given by
\[
\br H_{\tl u(t)}(p)=p(\br F(\br x,\tl u(t))).
\]
The third assertion of the theorem follows from this. The last assertion
follows from (\ref{PMP5}).
\end{proof}

Going back to Problem \ref{extendmin}, we can apply Pontryagin Maximum Principle for Mayer problem. Let $(\tl x(\cdot),\tl z(\cdot))$ be an admissible pair which minimizes Problem \ref{extendmin} and let $\br H_t:T^*M\times\Real\to\Real$ be the function defined by

\[
\br H_t(p,l)=p(F(x,\tl u(t)))+l\cdot L(x,\tl u(t)).
\]
By Theorem \ref{PMP}, there exists a curve $(\tl p(\cdot),\tl l(\cdot)):[0,1]\to T^*_{\tl x}
M\times\Real$ such that $\tl x(t)=\pi(\tl p(t))$ and

\begin{equation}\label{extendHam}
\left\{%
\begin{array}{ll}
    (\dot{\tl p},\dot{\tl l})=\overrightarrow{\br H}_t(\tl p,\tl l), \\
    (\tl p(1),\tl l(1))=(-\alpha,1), \\
    \br H_t(\tl p(t),\tl l(t)) =
    \min\limits_{u\in U}\left(\tl p(t)(F(\tl x(t),u))+\tl l(t)\cdot L(\tl x(t),u)\right)
\end{array}%
\right.
\end{equation}

From the first equation in (\ref{extendHam}), we get $\dot{\tl l}=0$. So, $\tl l(t)\equiv 1$. Therefore, (\ref{extendHam}) is
simplified to
\begin{equation}
\left\{%
\begin{array}{ll}
    \dot {\tl p}=\overrightarrow H_{\tl u}(\tl p),\\
    \tl p(1)=-\alpha,\\
    H_{\tl u}(\tl p(t),\tl P(t))=
    \min\limits_{u\in U}\left(\tl p(t)(F(\tl x(t),u))+L(\tl x(t),u)\right).
\end{array}%
\right.
\end{equation}

This finishes the proof of Theorem \ref{PMPB}.

\section*{Acknowledgment}

The second author would like to express deep gratitude to his
supervisor, Boris Khesin, who suggested to him the problem of
optimal mass transportation on subriemannian manifolds.

\end{document}